\documentclass[cupthm,crop,info]{CUP-JNL-ETS}%

\usepackage{graphicx}
\usepackage{multicol,multirow}
\usepackage{amsmath,amssymb,amsfonts}
\usepackage{mathrsfs}
\usepackage{rotating}
\usepackage{appendix}
\usepackage[numbers]{natbib}
\usepackage{amsthm}
\usepackage{lipsum}
\usepackage{enumitem}

\theoremstyle{cupplain}
\newtheorem{theorem}{Theorem}[section]
\newtheorem*{theorem*}{Theorem}
\newtheorem{lemma}[theorem]{Lemma}

\newtheorem{proposition}{Proposition}

\theoremstyle{cupdefinition}
\newtheorem{definition}{Definition}[section]
\theoremstyle{cupremark}
\newtheorem{remark}[theorem]{Remark}

\theoremstyle{cupproof}
\newtheorem{proof}{Proof}

\numberwithin{equation}{section}

\newcommand{\intersection}{\mathop{\bigcap}}
\newcommand{\union}{\mathop{\bigcup}}

\newcommand{\re}{\mathrm{Re}}
\newcommand{\im}{\mathrm{Im}}

\newcommand{\diam}{\mathrm{diam}}

\newcommand{\del}{\partial}
\newcommand{\norm}[1]{\left\lVert#1\right\rVert}
\newcommand{\htop}{h_{\mathrm{top}}}

\newcommand{\SL}{\mathrm{SL}}
\newcommand{\cl}{\mathrm{cl}}
\newcommand\blank{\hphantom{==}}
\newcommand\numberthis{\addtocounter{equation}{1}\tag{\theequation}}

\newcommand{\N}{\mathbb N}
\newcommand{\Z}{\mathbb Z}

\newcommand{\R}{\mathbb R}
\newcommand{\C}{\mathbb C}
\newcommand{\T}{\mathbb T}

\newcommand{\Ss}{\mathcal S}

\newcommand{\Ps}{\mathcal P}
\newcommand{\Ds}{\mathcal D}

\newcommand{\Fs}{\mathcal F}
\newcommand{\Us}{\mathcal U}

\newcommand{\Is}{\mathcal I}

\let\oldphi\phi
\let\phi\varphi

\let\epsilon\varepsilon

\let\tilde\widetilde

\let\hat\widehat

\jyear{2020}
\jdoi{10.1017/etds.2020.xx}
\begin{document}
	
	\begin{Frontmatter}
		
		\title{Thermodynamics of Smooth Models of Pseudo-Anosov Homeomorphisms}
		
		\author{\gname{Dominic} \sname{Veconi}}
		%\author{\gname{Author} \sname{Name2}}
		
		\address{\orgdiv{Department of Mathematics}, \orgname{Penn State University}, \orgaddress{\city{University Park}, \state{PA}, \postcode{16802}, \country{USA}}\\ (\email{dveconi@gmail.com})}
	
		\Received{\textup{29} December \textup{2019}}
		\Accepted[ and accepted in revised form]{\textup{7} November \textup{2020}}
		
		\maketitle
		
		\authormark{D. Veconi}
		\titlemark{Thermodynamics of smooth models of pseudo-Anosov homeomorphisms}
		
		\begin{abstract}
			We develop a thermodynamic formalism for a smooth realization of pseudo-Anosov surface homeomorphisms. In this realization, the singularities of the pseudo-Anosov map are assumed to be fixed, and the trajectories are slowed down so the differential is the identity at these points. Using Young towers, we prove existence and uniqueness of equilibrium states for geometric $t$-potentials. This family of equilibrium states includes a unique SRB measure and a measure of maximal entropy, the latter of which has exponential decay of correlations and the Central Limit Theorem.
		\end{abstract}
		
		

	\end{Frontmatter}

	\section[Introduction]{Introduction}
	In \cite{BThursty}, W. Thurston classified linear automomorphisms of the torus into three classes, according to the eigenvalues of the automorphism $A \in \SL(2,\Z)$: 
	\begin{itemize}
		\item Diagonalizable automorphisms with eigenvalues of modulus 1 (rotations);
		\item Nondiagonalizable automorphisms (Dehn twists); 
		\item Automorphisms with eigenvalues of modulus $\neq 1$ (Anosov diffeomorphisms). 
	\end{itemize}
	In this same work, Thurston went on to classify homeomorphisms of any surface up to isotopy class. The principle was quite similar, and is now known as the Nielson-Thurston classification of elements of mapping class groups. This is summarized in the following theorem: 
	\begin{theorem*}
		Let $M$ be a compact orientable surface, and let $f : M \to M$ be a homeomorphism. Then $f$ is isotopic to a homeomorphism $F$ satisfying exactly one of the following three conditions:
		\begin{itemize}
			\item $F$ is a rotation: There is an integer $n$ for which $F^n \equiv \mathrm{Id}$. 
			\item $F$ is reducible: There is a closed curve in $M$ which $F$ leaves invariant. 
			\item $F$ is pseudo-Anosov. 
		\end{itemize}
	\end{theorem*}
	Of these three isotopy classes, from a dynamical systems perspective, the pseudo-Anosov maps are the most interesting. The most familiar example of a pseudo-Anosov map is the Arnold ``cat map'' of the two-dimensional torus $\T^2$, which is in fact an Anosov diffeomorphism. No other surface admits an Anosov diffeomorphism, but pseudo-Anosov homeomorphisms of surfaces besides $\T^2$ form an analogy of Anosov maps to other surfaces. Like their Anosov cousins, pseudo-Anosov maps admit a pair of transverse foliations of the state space, and the map uniformly contracts points along the leaves of one foliation and uniformly dilates points along the leaves of the other. In the traditional definition of a pseudo-Anosov homeomorphism (see Section 2), the contraction and dilation factors are constant and inverses of each other, similarly to a hyperbolic toral automorphism such as the cat map. (Accordingly, these maps are often referred to as ``linear pseudo-Anosov maps'', e.g. \cite{GouLinearPA}.) The primary difference between Anosov and pseudo-Anosov maps is the presence of finitely many singularities in the foliations. These are points where three or more leaves of one of the foliations meet at a single point. These leaves are known as ``prongs'' of the singularity. The constant rate of contraction and expansion along the transverse foliations mean the map is globally smooth except at the singularities. Pseudo-Anosov homeomorphisms have found their way into almost every field of geometry, such as Teichm\"uller theory and algebraic geometry. However, the ergodic properties of globally smooth realizations of pseudo-Anosov maps remains a relatively undeveloped area of study. 
	
	In \cite{GerbKatPA}, M. Gerber and A. Katok produced a $C^\infty$ realization of pseudo-Anosov homeomorphisms by slowing down the trajectories near the isolated singularities. The result is a surface diffeomorphism that is uniformly hyperbolic away from a finite set of fixed-point singularities, but whose differential slows down to the identity at these fixed points, thus admitting Lyapunov exponents of zero. These smooth pseudo-Anosov models also admit continuous foliations whose leaves are smooth except at the fixed singular points. Pseudo-Anosov diffeomorphisms constructed in this way are analogues of the one-dimensional Manneville-Pomeau map of the unit interval to compact surfaces of arbitrary genus (see \cite{ManPom}), in that they admit finitely many fixed-point singularities where the differential slows down to the identity, but the map exhibits uniform hyperbolicity away from these singularities. 
	
	To discuss the ergodic properties of these pseudo-Anosov diffeomorphisms, we use techniques and results from thermodynamic formalism. Thermodynamic formalism has been used to study ergodic and geometric properties of several classes of nonuniformly hyperbolic and nonuniformly expanding maps. One objective of thermodynamic formalism is to determine the existence and uniqueness of probability measures known as \emph{Sinai-Ruelle-Bowen (SRB) measures}. These are invariant measures that admit positive Lyapunov exponents almost everywhere, and have absolutely continuous conditional measures on unstable submanifolds (see Section 4). They are also known as ``physical measures'', in the sense that the set of points $x \in M$ for which we have
	\[
	\lim_{n \to \infty} \frac 1 n \sum_{k=0}^{n-1} \phi\big(f^n(x)\big) = \int \phi \, d\mu \quad \textrm{for any } \phi \in C^0(M)
	\]
	has positive measure. More generally, one also may consider \emph{equilibrium measures} for a given potential $\phi \in C^0(M)$. Equilibrium measures are mathematical generalizations of Gibbs distrubtions in statistical physics, which minimize the Helmholtz free energy of a physical system. Within thermodynamic formalism, Helmholtz free energy is replaced with the topological pressure $P_f(\phi) = \sup \left\{ h_\mu(f) + \int\phi\,d\mu : \mu \in \mathcal M_f \right\}$, where $h_\mu(f)$ is the metric entropy of $f$ with respect to $\mu$, and $\mathcal M_f$ is the space of $f$-invariant Borel probability measures on the manifold $M$. Equilibrium measures, in other words, are invariant probability measures that maximize the sum of the metric entropy of $f$ and the space average of $\phi$ with respect to $\mu$. The most important two equilibrium measures are SRB measures (for which the potential is the negative log of the unstable Jacobian, or $\phi_1(x) = -\log\det\left| Df_x|_{E^u(x)}\right|$), and measures of maximal entropy (for which the potential is $\phi_0 \equiv 0$). 
	
	One of the earliest applications of thermodynamic formalism was in studying the ergodic theory of uniformly hyperbolic and Axiom A diffeomorphisms (e.g. \cite{Bowen}). Since then, the theory of thermodynamic formalism has proven useful in other contexts. For example, the one-dimensional Manneville-Pomeau maps $f : [0,1] \to [0,1]$, defined by $f(x) = x(1+ax^{\alpha}) \mod 1$ for $a>0$, $\alpha>0$, have been extensively studied as classic examples of one-dimensional nonuniformly expanding maps (see, e.g., \cite{PolWei99}, as well as \cite{Ven18} for some recent work on the infinite ergodic theory of Manneville-Pomeau maps). Additionally, in \cite{CPZ19}, V. Climenhaga, Y. Pesin, and A. Zelerowicz proved existence of equilibrium measures for a broad class of potential functions in the partially hyperbolic setting. These equilibrium measures include, in particular, a unique measure of maximal entropy and a unique SRB measure. Finally, in \cite{BCS19}, J. Buzzi, S. Crovisier, and O. Sarig showed that any surface diffeomorphism admits at most finitely many ergodic measures of maximal entropy, and that there is a unique such measure in the topologically transitive case. Our results are a special instance of this setting, and develop further statistical and ergodic properties of the measure of maximal entropy and other equilibrium states. 
	
	%Specifically, given a pseudo-Anosov diffeomorphism $g$ of a compact surface $M$, we consider the family of geometric $t$-potentials 
	%$\varphi_t(x) = −t log 􏰈Dg|E^u(x)$ where $E^u(x)$ is a stable subspace at the point $x$ and $t\in\mathbb{R}$. Our main result, Theorem 4.1 claims that there is a number $t_0<0$ such for every $t\in (t_0,1)$ there is a unique equilibrium measure $\mu_t$ for $\varphi_t$ which is Bernoulli, has exponential decay of correlations and satisfies the Central Limit Theorem with respect to a class of functions containing all H\"older continuous functions on $M$. We also show hat the pressure function $t\to P_g(\varphi_t)$ is real analytic in the open interval 
	%$(t_0,1)$. 
	
	%Since the pseudo-Anosov diffeomorphism $g$ is topologically conjugate to a pseudo-Anosov homeomorphism $f$, their topological entropies agree and since $f$ has a unique measure of maximal entropy, so does $g$. However, as a corollary of Theorem 4.1, we obtain a sufficiently complete description of ergodic properties of the measure $\mu_0$ of maximal entropy for $g$. 
	
	%We also prove that the map $f$ has a unique SRB measure and we describe its ergodic properties. Let us stress that for $t=1$ a phase transition occurs: there are finitely many ergodic equilibrium measures for $\varphi_1$, one is the SRB measure and others are the delta measures at the singularities.
	
	In this paper, we effect a thermodynamic formalism for these pseudo-Anosov diffeomorphisms. Specifically, given a pseudo-Anosov diffeomorphism $g$ of a compact surface $M$, we consider the family of geometric $t$-potentials $\phi_t(x) = -t\log\left|Dg|_{E^u(x)}\right|$ parametrized by $t \in \R$, where $E^u(x)$ is the stable subspace at the point $x \in M$. Our main result, Theorem \ref{main-theorem}, claims that there is a number $t_0 < 0$ such that for every $t \in (t_0, 1)$, there is a unique equilibrium measure $\mu_t$ for $\phi_t$ that is Bernoulli, has exponential decay of correlations, and satisfies the Central Limit Theorem with respect to a class of functions containing all H\"older continuous functions on $M$. We also show that the pressure function $t \mapsto P_g(\phi_t)$ is real analytic in the open interval $(t_0, 1)$. Since the pseudo-Anosov diffeomorphism $g$ is topologically conjugate to a pseudo-Anosov homeomorphism $f$, their topological entropies agree, and since $f$ has a unique measure of maximal entropy, so does $g$. We denote this measure $\mu_0$, for the potential $\phi_0 \equiv 0$. As a corollary to Theorem \ref{main-theorem}, we obtain a thorough description of the statistical properties of $\mu_0$. Furthermore, we prove that the map $g$ has a unique SRB measure, and we describe its ergodic properties. We emphasize that a phase transition occurs at $t=1$: in addition to the SRB measure, there is a family of ergodic equilibrium measures for $\phi_1$ composed of convex combinations of Dirac measures at the singularities.  %we prove existence and uniqueness of a measure of maximal entropy and equilibrium states for a range of geometric $t$-potentials, $\phi_t(x) = -t\log\left|DG|_{E^u(x)}\right|$. Additionally, we show the measure of maximal entropy has exponential decay of correlations and satisfies the central limit theorem with respect to a class of functions including all H\"older potentials. Furthermore, our methods demonstrate that these smooth models of pseudo-Anosov maps admit a unique SRB measure. The primary difficulty in proving these results stems from the fact that the differential of Gerber's smooth pseudo-Anosov models do not vary H\"older continuously, which results in the potential $\phi_t$ not varying H\"older continuously either. This may result in phase transitions in the system. In particular, the potential $\phi_1(x) = -\log\left|DG|_{E^u(x)}\right|$ admits two classes of equilibrium measures. The first class contains linear combinations of the Dirac measures at the singularities, and the second class contains only an SRB measure. 
	
	%The techniques we employ to establish our results are similar to those used by Pesin, Senti and Zhang in [11] to effect thermodynamic formalism of the Katok map. The latter is an area preserving diffeomorphism of the torus with non-zero Lyapunov exponents. Similarly to the smooth pseudo-Anosov models, the Katok map is obtained by slowing down trajectories near the origin to produce an indifferent fixed point (i.e. a fixed point of the map whose differential is equal to the identity). However, there are substantial differences between the Katok map of the torus and Gerber-Katok's smooth pseudo-Anosov models: 
	%\begin{enumerate}
	%	\item the Katok map acts on the torus and thus can be lifted to $\mathbb{R}^2$, while pseudo-Anosov maps do not admit a lift to $\mathbb{R}^2$. The possibility of lifting the Katok map to $\mathbb{R}^2$ simplifies the analysis of the map in [11], and some adjustments to this argument is required to carry out similar analysis of globally smooth pseudo-Anosov diffeomorphisms. 
	%	\item whereas the slow-down function used in the Katok map depends only on the radius of the slowed-down neighborhood, the choice of slow-down function for the pseudo-Anosov homeomorphisms depends on the number of prongs of the singularity; this effects in a substantial way the analysis of the behavior of trajectories near the singularities. 
	%\end{enumerate} 
	The techniques we employ to establish our results are similar to those used by Y. Pesin, S. Senti, and K. Zhang in \cite{PSZ17} to effect thermodynamic formalism of the Katok map. The latter is an area preserving diffeomorphism of the torus with non-zero Lyapunov exponents. Similarly to the smooth pseudo-Anosov models, the Katok map is obtained by slowing down trajectories near the origin to produce an indifferent fixed point (i.e. a fixed point of the map whose differential is equal to the identity). However, there are substantial differences between the Katok map of the torus and the Gerber-Katok smooth pseudo-Anosov models. These include:
	\begin{itemize}
		\item The Katok map acts on the torus, and thus can be lifted to $\R^2$, while pseudo-Anosov maps do not in general admit a lift to $\R^2$. The lift of the Katok map to $\R^2$ plays an essential role in simplifying the analysis in \cite{PSZ17}, and some adjustments to this argument are required to carry out similar analysis of globally smooth pseudo-Anosov diffeomorphisms. 
		\item The foliations of pseudo-Anosov diffeomorphisms are singular. In particular, the singularities do not admit a locally stable or unstable subspace forming a curve, but rather forming the prongs that meet at the singularity. Furthermore, one cannot use coordinate charts whose interiors contain the singularities if the coordinates correspond to the stable and unstable foliations. Instead, the analysis must be performed in stable and unstable sectors whose vertices are the singularities (see Section 3).
		\item Whereas the slow-down function used to construct the Katok map depends only on the radius of the slowed-down neighborhood, the choice of slow-down function of the pseudo-Anosov homeomorphism depends on the number of prongs of the singularity. This affects the analysis of the behavior of the trajectories near the singularities.
	\end{itemize}

	The development of thermodynamics of the Katok map in \cite{PSZ17} uses the technology of Young diffeomorphisms, which are generalizations of hyperbolic maps. The definition of Young diffeomorphisms relies on hyperbolicity of an induced map on a small subset of the state space with local hyperbolic product structure. This induced map can be carried over to a derived dynamical system on the corresponding Rokhlin tower. The thermodynamics of Young diffeomorphisms have been thoroughly investigated in \cite{PSZTowers} and in \cite{FZTowers}. Young towers have been used to study thermodynamic and ergodic properties of a variety of nonuniformly hyperbolic dynamical systems (see \cite{CP17}), including almost Anosov toral diffeomorphisms (see \cite{me}). %Young diffeomorphisms are a generalization of hyperbolic maps. Their definition relies on hyperbolicity of an induced map on a small subset of the state space with local hyperbolic product structure. This induced map can be carried over to a derived dynamical system on a hyperbolic Rokhlin tower. The thermodynamics of Young diffeomorphisms and hyperbolic towers have been thoroughly investigated in \cite{PSZTowers} and in \cite{FZTowers}. In addition to thermodynamic properties of the Katok map, Young towers have been used to study thermodynamic and ergodic properties of a variety of other systems, including almost Anosov toral diffeomorphisms in \cite{me}, as well as other nonuniformly hyperbolic systems in \cite{CP17}. Using the tools developed in these works, we are able to effect a thermodynamic formalism for smooth models of pseudo-Anosov maps. 
	
	%This paper is structured as follows. In Section 2, we define pseudo-Anosov homeomorphisms and discuss some of their dynamical properties, including measure invariance and Markov partitions. In Section 3, we describe the smooth models of pseudo-Anosov homeomorphisms  and state some important dynamical and topological properties of these maps. We state our main results in Section 4. Section 5 is devoted to the study of dynamics near the singularities and include some technical calculations needed to prove our main result. Some of these calculations are similar to the ones performed in Section 5 of [11] but require some modifications and adjustments. Section 6 gives a brief survey of the thermodynamic properties of Young diffeomorphisms and inducing schemes we will be using. Section 7 proves that our smooth models of pseudo-Anosov homeomorphisms are Young diffeomorphisms, and finally Section 8 uses this fact to prove our main results.
	
	This paper is structured as follows. In Section 2, we define pseudo-Anosov homeomorphisms and discuss some of their dynamical properties, including measure invariance and Markov partitions. In Section 3, we describe the smooth models of pseudo-Anosov homeomorphisms and state some important dynamical and topological properties of these maps. We state our main results in Section 4. Section 5 is devoted to the study of dynamics near the singularities and include some technical calculations needed to prove our main result. Some of these calculations are similar to the ones performed in Section 5 of \cite{PSZ17} but require some modifications and adjustments. Section 6 gives a brief survey of the thermodynamic properties of Young diffeomorphisms and inducing schemes we will be using. Section 7 proves that our smooth models of pseudo-Anosov homeomorphisms are Young diffeomorphisms, and finally Section 8 uses this fact to prove our main results.
	
	\section{Preliminaries}
	We begin with a discussion on measured foliations of a compact two-dimensional $C^\infty$ Riemannian manifold $M$, where we assume $M$ is without boundary. Our exposition is adapted from the presentation in \cite{BaPeNUH}, Section 6.4. For the reader's convenience, we have restated their exposition here and have included additional details and remarks on the notation concerning pseudo-Anosov maps and their behavior.
	
	%\begin{definition}
	%	A \emph{foliation with singularities} is a collection $\Fs$ of curves on the manifold $M$ homeomorphic to $\R$ or $\Sbb^1$, along with a collection of \emph{singularities} $x_1, \ldots, x_m \in M$, for which the following properties hold: 
	%	\begin{enumerate}
	%		\item $\Fs$ is a $C^\infty$ foliation of $M \setminus \{x_1, \ldots, x_m\}$; 
	%		\item 
	%	\end{enumerate}
	%\end{definition}

	\begin{definition}
		A \emph{measured foliation with singularities} is a triple $(\Fs, S, \nu)$, where: 
		\begin{itemize}
			\item $S = \{x_1, \ldots, x_m\}$ is a finite set of points in $M$, called \emph{singularities}; 
			\item $\Fs = \tilde \Fs \uplus \Ss$ is a partition of $M$, where $\Ss$ is a partition of $S$ into points and $\tilde \Fs$ is a smooth foliation of $M \setminus S$;
			\item $\nu$ is a \emph{transverse measure}; in other words, $\nu$ is a measure defined on each curve on $M$ transverse to the leaves of $\tilde \Fs$;
		\end{itemize}
		and the triple satisfies the following properties:
		\begin{enumerate} 
			\item There is a finite atlas of $C^\infty$ charts $\oldphi_k : U_k \to \C$ for $k = 1, \ldots, \ell$, $\ell \geq m$. 
			\item For each $k = 1, \ldots, m$, there is a number $p = p(k) \geq 3$ of elements of $\tilde\Fs$ meeting at $x_k \in S$ (these elements are called \emph{prongs} of $x_k$) such that: 
			\begin{enumerate}[label=(\alph*)]
				\item $\oldphi_k(x_k) = 0$ and $\oldphi_k(U_k) = D_{a_k} := \{z \in \C : |z| \leq a_k\}$ for some $a_k > 0$; 
				\item if $C \in \tilde\Fs$, then the components of $C \cap U_k$ are mapped by $\oldphi_k$ to sets of the form 
				\[
				\left\{z \in \C: \im\left(z^{p/2}\right) = \mathrm{constant} \right\} \cap \oldphi_k(U_k);
				\]
				\item the measure $\nu|U_k$ is the pullback under $\oldphi_k$ of $$\left| \im\left(dz^{p/2}\right)\right| = \left| \im\left(z^{(p-2)/2} dz\right)\right|.$$
			\end{enumerate}
			\item For each $k > m$, we have: 
			\begin{enumerate}[label=(\alph*)]
				\item $\oldphi_k(U_k) = (0, b_k) \times (0,c_k) \subset \R^2 \approx \C$ for some $b_k, c_k > 0$; 
				\item If $C \in \tilde\Fs$, then components of $C \cap U_k$ are mapped by $\oldphi_k$ to lines of the form
				\[
				\{z \in \C : \im \,z = \mathrm{constant}\} \cap \oldphi_k(U_k).
				\] 
				\item The measure $\nu|U_k$ is given by the pullback of $|\im \,dz|$ under $\oldphi_k$. 
			\end{enumerate}
		\end{enumerate}
	\end{definition}
	
	An archetypal singularity with $p=3$ prongs is shown in Figure 1. 
	
	%\begin{remark}
	%	The partition $\Fs$ is an example of what is known as a ``singular foliation''. We do not give the formal definition of a singular foliation here, but we will use this terminology when referring to the partition $\Fs$ and similar partitions of this form.  
	%\end{remark}
	
	\begin{remark}
		Henceforth, we refer to the $C^\infty$ curves that are elements of $\Fs$ as ``leaves (of the foliation)''; in particular, despite the technical fact that the singleton sets of singularities $\{x_1\}, \ldots, \{x_k\}$ are elements of $\Fs$, we do not refer to these points when we refer to ``leaves of the foliation''. 
	\end{remark}
	
	\begin{remark} The transverse measure $\nu$ is not a measure on $M$ itself, in the measure-theoretic sense of the word. What $\nu$ is measuring is the ``distance traveled'' transverse to the leaves of the foliation, similarly to how the 1-form $dx$ measures distance traveled transverse to the leaves $\{x=x_0\}$. To make this more explicit, properties (2) and (3) in the above definition ensure that $\nu$ is holonomy-invariant. In particular, if $\gamma$ and $\gamma'$ are isotopic curves in $M \setminus S$ transverse to the leaves of $\Fs$, and the initial points of $\gamma$ and $\gamma'$ lie in the same leaf $\Fs_0$ and the terminal points lie in the same leaf $\Fs_1$, then $\nu(\gamma) = \nu(\gamma')$. %We refer to transverse measures for a foliation that are holonomy-invariant as \emph{transverse invariant measures}. Note this does not imply invariance with respect to any dynamical system over $M$; such invariance will be explicitly stated as such henceforth. 
	\end{remark}
	
	\begin{figure}
		\centering
		\includegraphics[width=0.6\textwidth]{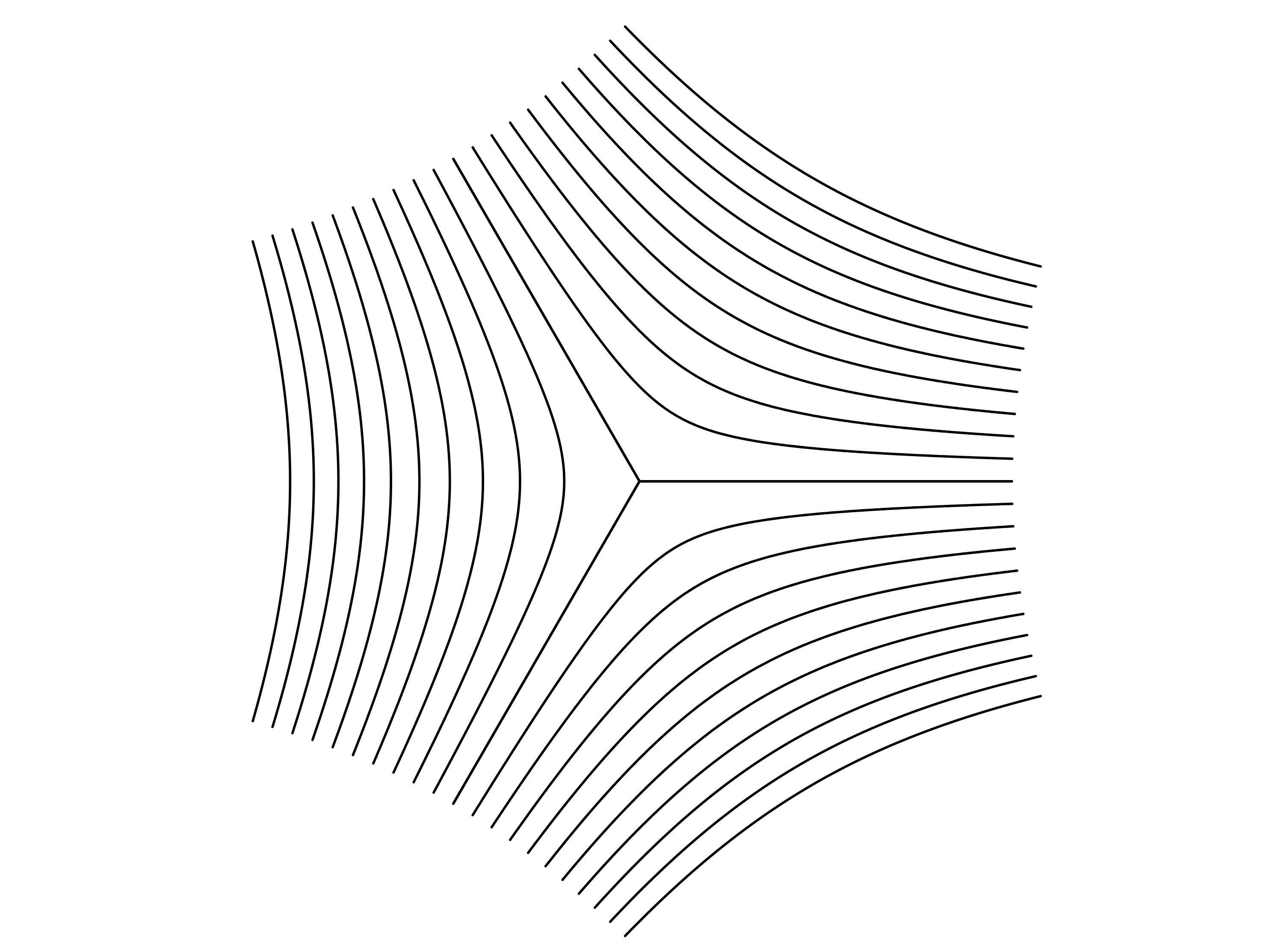}
		\caption{A 3-pronged singularity of a measured foliation with singularities.}
	\end{figure}
	
	%The defining property of pseudo-Anosov maps is that they admit two measured foliations with singularities, whose curve leaves intersect transversally and whose singularities coincide. We now
	
	\begin{definition}\label{PAH-def}
		A surface homeomorphism $f$ of a manifold $M$ is \emph{pseudo-Anosov} if there are measured foliations with singularities $(\Fs^s, S, \nu^s)$ and $(\Fs^u, S, \nu^u)$ (with the same finite set of singularities $S = \{x_1, \ldots, x_m\}$) and an atlas of $C^\infty$ charts $\oldphi_k : U_k \to \C$ for $k = 1, \ldots, \ell$, $\ell > m$, satisfying the following properties: 
		\begin{enumerate}
			\item $f$ is differentiable, except on $S$.
			\item For each $x_k \in S$, $\Fs^s$ and $\Fs^u$ have the same number $p(k)$ of prongs at $x_k$. 
			\item The leaves of $\Fs^s$ and $\Fs^u$ intersect transversally at nonsingular points.
			\item Both measured foliations $\Fs^s$ and $\Fs^u$ are $f$-invariant.
			\item There is a constant $\lambda > 1$ such that 
			\[
			f(\Fs^s, \nu^s) = (\Fs^s, \nu^s/\lambda) \quad \textrm{and} \quad f(\Fs^u, \nu^u) = (\Fs^u, \lambda \nu^u).
			\]
			\item For each $k=1, \ldots, m$, we have $x_k \in U_k$, and $\oldphi_k : U_k \to \C$ satisfies: 
			\begin{enumerate}[label=(\alph*)]
				\item $\oldphi_k(x_k) = 0$ and $\oldphi_k(U_k) = D_{a_k}$ for some $a_k > 0$; 
				\item if $C$ is a curve leaf in $\Fs^s$, then the components of $C \cap U_k$ are mapped by $\oldphi_k$ to sets of the form 
				\[
				\left\{z \in \C: \re\left(z^{p/2}\right) = \mathrm{constant}\right\}\cap D_{a_k};
				\]
				\item if $C$ is a curve leaf in $\Fs^u$, then the components of $C \cap U_k$ are mapped by $\oldphi_k$ to sets of the form 
				\[
				\left\{z \in \C : \im\left(z^{p/2}\right) = \mathrm{constant} \right\} \cap D_{a_k};  
				\]
				\item the measures $\nu^s|U_k$ and $\nu^u|U_k$ are given by the pullbacks of $$\left|\re\left(dz^{p/2}\right)\right| = \left|\re\left(z^{(p-2)/2} dx \right)\right|$$
				and $$ \left|\im\left(dz^{p/2}\right)\right| = \left|\im\left(z^{(p-2)/2} dx \right)\right|$$
				under $\oldphi_k$, respectively. 
			\end{enumerate}
			\item For each $k > m$, we have: 
			\begin{enumerate}[label=(\alph*)]
				\item $\oldphi_k(U_k) = (0, b_k) \times (0,c_k) \subset \R^2 \approx \C$ for some $b_k, c_k > 0$; 
				\item If $C$ is a curve leaf in $\Fs^s$, then components of $C \cap U_k$ are mapped by $\oldphi_k$ to lines of the form
				\[
				\{z \in \C : \re \,z = \mathrm{constant}\} \cap \oldphi_k(U_k);
				\] 
				\item If $C$ is a curve leaf in $\Fs^u$, then components of $C \cap U_k$ are mapped by $\oldphi_k$ to lines of the form
				\[
				\{z \in \C : \im \,z = \mathrm{constant}\} \cap \oldphi_k(U_k);
				\] 
				\item the measures $\nu^s|U_k$ and $\nu^u|U_k$ are given by the pullbacks of $|\re\,dz|$ and $|\im \,dz|$ under $\oldphi_k$, respectively. 
			\end{enumerate}
		\end{enumerate}
		For $k = 1, \ldots, m$, we call the neighborhood $U_k \subset M$ described in part (6) of this definition a \emph{singular neighborhood}, and for $k > m$, we call $U_k$ a \emph{regular neighborhood}. (See Figure 2.)
	\end{definition} 
	
	\begin{figure}
		\centering
		\includegraphics[width=0.6\textwidth]{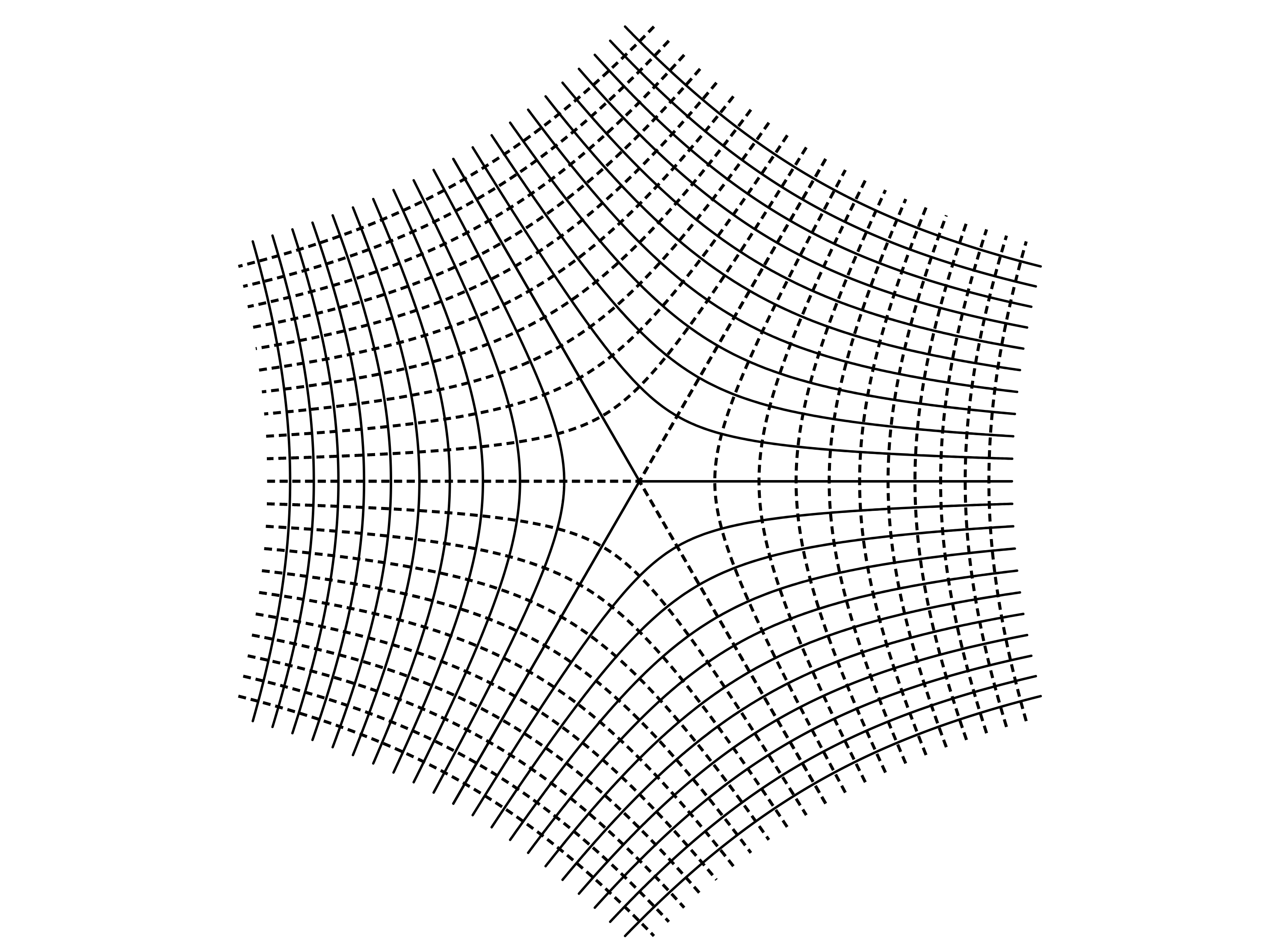}
		\caption{A singular neighborhood with a 3-pronged singularity. The solid lines and broken lines respectively represent the stable and unstable foliations $\Fs^s$ and $\Fs^u$, for example.}
	\end{figure}
	
	\begin{remark}\label{PAH uniformly hyperbolic}
		The notation $f(\Fs^u, \nu^u) = (\Fs^u, \lambda\nu^u)$ means two things. First, it means that if $\gamma$ is a subset of a leaf of $\Fs^u$, then so is $f(\gamma)$, and in particular, so is $f^{-1}(\gamma)$. Second, it means if $\gamma$ is an open interval in $\Fs^s$, or more generally any arc in $M$ transverse to the foliation $\Fs^u$, then $\nu^u\left(f^{-1}(\gamma)\right) = \lambda \nu^u(\gamma)$. That is, $f_* \nu^u = \lambda\nu^u$, with $f_* \nu^u$ the pushforward transverse measure. Likewise for the notation $f(\Fs^s, \nu^s) = (\Fs^s, \nu^s/\lambda)$. So points on the same $\Fs^s$-leaf contract in the $\nu^u$-measure by a factor of $\lambda$, and points on the same $\Fs^u$-leaf dilate in the $\nu^s$-measure by a factor of $\lambda$. 
	\end{remark}
	
	\begin{remark}
		Since $f$ is a homeomorphism, $f$ permutes the singularities; that is, the singular set $S$ is $f$-invariant. However, our arguments assume the singularities are fixed under the pseudo-Anosov homeomorphism. If the singularities are not fixed points, one could consider an appropriate iterate of $f$ and study the dynamics of this iterate, arriving at the same results. 
	\end{remark}
	
	We state a few important properties of pseudo-Anosov homeomorphisms we will use over the course of our arguments. 
	
	\begin{proposition}\label{PAH-differential}
		Let $f : M \to M$ be a pseudo-Anosov homeomorphism. For $x \in M \setminus S$, $T_x M = T_x \Fs^s(x) \oplus T_x \Fs^u(x)$, and in these coordinates, $Df_x(\xi^s, \xi^u) = \left( \xi^s/\lambda, \lambda \xi^u\right)$, where $\xi^s$ and $\xi^u$ are nonzero vectors in $T_x \Fs^s(x)$ and $T_x \Fs^u(x)$, $\Fs^s(x)$ and $\Fs^u(x)$ represent the curve containing $x$ in the respective foliation, and $\lambda$ is the dilation factor for $f$. 
	\end{proposition}
	
	\begin{proof}
		This follows immediately from the definition of pseudo-Anosov diffeomorphisms after a calculation in coordinates (see Remark \ref{PAH uniformly hyperbolic}). 
	\end{proof}
	
	\begin{proposition}[see \cite{FLP79}, Expos\'e 10]\label{PAH measure}
		A pseudo-Anosov surface homeomorphism $f : M \to M$ preserves a smooth invariant probability measure $\nu$ defined locally as the product of $\nu^s$ on $\Fs^u$-leaves with $\nu^u$ on $\Fs^s$-leaves. In any coordinate chart of $M$, this probability measure $\nu$ has a density with respect to the measure induced by the Lebesgue measure on $\R^2$, and this density vanishes at singularities. 
	\end{proposition}
	
	\begin{proposition}[see \cite{FLP79}, Expos\'e 10]\label{pseudo-Anosov Markov}
		Every pseudo-Anosov homeomorphism of a surface $M$ admits a finite Markov partition of arbitrarily small diameter. Conjugated to the symbolic system induced by this Markov partition, with the measure $\nu$ as in the preceding proposition, $(M, f, \nu)$ is Bernoulli. 
	\end{proposition}
	
	\section{Pseudo-Anosov diffeomorphisms}

Generally speaking, pseudo-Anosov homeomorphisms as defined in Definition \ref{PAH-def} are differentiable everywhere except at the singularities $x_k$ with $p(k) \geq 3$. This is a consequence of the fact that $f$ contracts (resp. expands) points in the stable (resp. unstable) leaves of the foliation, so the differential of $f$ cannot possibly be linear at the singularities.

In this section, we construct a surface diffeomorphism $g : M \to M$ that is topologically conjugate to the pseudo-Anosov homeomorphism $f$, and whose differential at the singularity is the identity. (Since we assume the singularities are fixed, this is a reasonable statement.) 

Before proceeding with the construction, we point out that some literature refers to the maps defined in Definition \ref{PAH-def} as ``pseudo-Anosov diffeomorphisms'', despite the fact that these maps are not differentiable at the singularities. To avoid any confusion, we reserve the word ``diffeomorphism'' only for those maps that are differentiable on all of $M$, and use the phrase ``pseudo-Anosov homeomorphism'' for the maps described in Definition \ref{PAH-def}. 

Let $x_k \in S$, let $p = p(x_k)$, and let $\oldphi_k : U_k \to \C$ be the chart described in part (6) of Definition (\ref{PAH-def}). The \emph{stable} and \emph{unstable prongs} at $x_k$ are the leaves $P^s_{kj}$ and $P^u_{kj}$, $j = 0, \ldots, p-1$ of $\Fs^s$ and $\Fs^u$, respectively, whose endpoints meet at $x_k$. Locally, they are given by:
\begin{align*}
P^s_{kj} &= \oldphi_k^{-1} \left\{\rho e^{i\tau} : 0 \leq \rho < a_k, \: \tau = \frac{2j+1}p \pi \right\}, \\
\textrm{and}\quad P^u_{kj} &= \oldphi_k^{-1} \left\{\rho e^{i\tau} : 0 \leq \rho < a_k, \: \tau = \frac{2j}p \pi \right\}.
\end{align*}
For simplicity, assume $f(P^s_{kj}) \subseteq P^s_{kj}$ for all $j = 1, \ldots, p$. Furthermore, we define the \emph{stable} and \emph{unstable sectors} at $x_k$ to be the regions in $U_k$ bounded by the stable (resp. unstable) prongs: 
\begin{align*}
S^s_{kj} &= \oldphi_k^{-1} \left\{\rho e^{i\tau} : 0 \leq \rho < a_k, \: \frac{2j-1}p \pi \leq \tau \leq \frac{2j+1}p \pi \right\}, \\
\textrm{and}\quad S^u_{kj} &= \oldphi_k^{-1} \left\{\rho e^{i\tau} : 0 \leq \rho < a_k, \: \frac{2j}p \pi \leq \tau \leq \frac{2j+2}p \pi \right\}.
\end{align*}
The strategy for creating our diffeomorphism $g$ is adapted from section 6.4.2 of \cite{BaPeNUH}. In each stable sector, we apply a ``slow-down'' of the trajectories, followed by a change of coordinates ensuring the resulting diffeomorphism $g$ preserves the measure induced by a convenient Riemannian metric. 

Let $F : \C \to \C$ be the map $s_1 + is_2 \mapsto \lambda s_1 + is_2/\lambda$. Note $F$ is the time-1 map of the vector field $V$ given by 
\[
\dot{s}_1 = (\log\lambda)s_1, \quad \dot{s}_2 = -(\log\lambda) s_2.
\]
Let $0 < r_1 < r_0 < \min\{a_1, \ldots, a_\ell\}$, and define $\tilde r_0$ and $\tilde r_1$ by $\tilde r_j = (2/p)r^{p/2}_j$ for $j=0,1$ and for each $p = p(k)$. Define a ``slow-down'' function $\Psi_p$ for the $p$-pronged singularity on the interval $[0, \infty)$ so that: 
\begin{enumerate}[label=(\alph*)]
	\item $\Psi_p(u) = (p/2)^{(2p-4)/p} u^{(p-2)/p}$	for $u \leq \tilde r_1^2$; 
	\item $\Psi_p$ is $C^\infty$ except at $0$; 
	\item $\dot\Psi_p(u) \geq 0$ for $u > 0$; 
	\item $\Psi_p(u) = 1$ for $u \geq \tilde r_0^2$.
\end{enumerate}
Consider the vector field $V_{\Psi_p}$ on $D_{\tilde r_0} \subset \C$ defined by 
\begin{equation}\label{slowdown vector field}
\dot s_1 = (\log \lambda) s_1 \Psi_p \left(s_1^2 + s_2^2\right) \quad \textrm{and} \quad \dot s_2 = -(\log \lambda) s_2 \Psi_p\left(s_1^2 + s_2^2\right).
\end{equation}
Let $G_p$ be the time-1 map of the vector field $V_{\Psi_p}$. Assume $r_1$ is chosen to be small enough so that $G_p = F$ on a neighborhood of the boundary of $D_{\tilde r_0}$, and assume $r_0$ is chosen to be small enough so that the open neighborhood $\mathcal U_0 := \union_{k=1}^m \oldphi_k^{-1}\left(D_{r_0}\right)$ of $S$ is disjoint from the open set $\union_{k={m+1}}^\ell \oldphi_k^{-1}\left(D_{a_k}\right)$. We also define the open neighborhood $\tilde\Us_0 := \union_{k=1}^m \oldphi_k^{-1}\left(D_{\tilde r_0}\right) \subset \Us_0$, as well as $\Us_1$ and $\tilde \Us_1$ defined analogously with $D_{r_1}$ and $D_{\tilde r_1}$ respectively. 

Let $\tilde a_k = (2/p) a_k^{p/2}$, and define the coordinate change $\Phi_{kj} : \oldphi_k S^s_{kj} \to \left\{z : \re z \geq 0 \right\} \cap D_{\tilde a_k}$ by 
\[
\Phi_{kj}(z) = (2/p)z^{p/2} = w = s_1 + is_2.
\]
Define $g : M \to M$ by $g(x) = f(x)$ for $x \not\in \Us_0$ and meanwhile for $1 \leq k \leq m$, $1 \leq j \leq p(k)$, define $g$ on each sector $S^s_{kj}\cap\oldphi_k^{-1}\left(D_{r_0}\right)$ by 
\[
g(x) = \oldphi_k^{-1} \Phi_{kj}^{-1} G_p \Phi_{kj} \oldphi_k(x). 
\]

\begin{proposition}[see \cite{BaPeNUH}]\label{f-g conjugacy}
	The map $g$ defined above is well-defined on the unstable prongs and singularity. It is in fact a diffeomorphism topologically conjugate to $f$, and for any $\epsilon > 0$, $r_0$ and $r_1$ can be chosen so that $\norm{f - g}_{C^0} < \epsilon$. In particular, $g$ admits a Markov partition of arbitrarily small diameter. 
\end{proposition}

Next we define a Riemannian metric $\zeta = \langle \cdot, \cdot \rangle$ on $M \setminus S$ with respect to which the map $g$ is invariant. In the stable sector $S^s_{kj} \cap \oldphi_k^{-1}(D_{\tilde a_k})$, we consider the coordinates $w = s_1 + is_2$ given by $\Phi_{kj} \circ \oldphi_k$ defined above. Outside of this neighborhood, we use the coordinates $z = s_1 + is_2$. In both sets of coordinates, the stable and unstable transversal measures are $\nu^s = |ds_1|$ and $\nu^u = |ds_2|$. On stable sectors in $M \setminus S$, we define the Riemannian metric $\zeta$ to be the pullback of $\left(ds_1^2 + ds_2^2\right)/\Psi_p\left(s_1^2 + s_2^2\right)$ under $\Phi_{kj} \circ \oldphi_k$. In regular neighborhoods $(U_k, \oldphi_k)$, we define $\zeta = \oldphi_k^* \left(ds_1^2 + ds_2^2\right)$. Since $\tilde r_0$ is chosen so that $\oldphi_k^{-1}\left(D_{\tilde r_0}\right)$ is disjoint from regular neighborhoods, and $\Psi_p(u) \equiv 1$ for $u \geq \tilde r_0^2$, $\zeta$ is consistently defined on chart overlaps. One can further show that $\zeta$ agrees with the Euclidean metric in $\oldphi_k^{-1}\left(D_{\tilde r_0}\right)$. So $\zeta$ can be extended to a Riemannian metric on all of $M$. 

\begin{proposition}[see \cite{BaPeNUH}]\label{PAD measure}
	Letting $z = t_1 + i t_2$ be the coordinates given by $(\oldphi_k, U_k)$, $1 \leq k \leq m$, the Riemannian metric $\zeta$ is actually the Euclidean metric $dt_1^2 + dt_2^2$. In particular, the diffeomorphism $g : M \to M$ is $\mu_1$-area preserving, where $\mu_1$ is the volume determined by $\zeta$. 
\end{proposition}

For stable sectors $S^s_{kj}$, we use the coordinates $w = \Phi_{kj}^s(z) = s_1 + is_2$, and in regular neighborhoods $U_k$, $k \geq m$, we use the coordinates $z = s_1 + is_2$. Then $s_1$ represents the coordinate in the unstable foliation, and $s_2$ is the coordinate in the stable foliation. Define the coordinates $(\xi_1, \xi_2)$ in each tangent space $T_xM$, $x \in M \setminus S$, to be the coordinates with respect to 
\begin{equation}\label{sector-coords}
\left(\Phi_{kj} \circ \oldphi_k\right)^{-1}_* \left( \Psi_p\left(s_1^2 + s_2^2\right) \frac{\del}{\del s_i}\right), \quad i=1,2
\end{equation}
in each stable sector, and with respect to $\left(\oldphi_k\right)^{-1}_*\left(\del/\del s_i\right)$, $i=1,2$, in each regular neighborhood.  For $x \in M \setminus S$, let $C_x^+$ be the cone in $T_x M$ bounded by the lines $\xi_1 = \pm \xi_2$, respectively, and contains the tangent line to the $\Fs^u$ leaf through $x$. Respectively define $C_x^-$ to be the cone containing the $\Fs^s$ leaf. 
\begin{proposition}[see \cite{BaPeNUH}]\label{regular-cones}
	For $x \in M \setminus S$, the cones $C_x^+, C_x^-$ satisfy the following: 
	\begin{enumerate}[label=(\alph*)]
		\item $C_x^+$ and $C_x^-$ depend continuously on $x \in M \setminus S$; 
		\item $C_x^+$ (resp. $C_x^-$) is strictly invariant under $Dg$ (resp. $Dg^{-1}$) on $x \in M \setminus S$; 
		\item For each $x \in M \setminus S$, the intersections 
		\[
		E^u(x) := \intersection_{n = 0}^\infty Dg^n C^+_{g^{-n}(x)} \quad \textrm{and} \quad E^s(x) := \intersection_{n=0}^\infty Dg^{-n} C^-_{g^{n}(x)}
		\]
		are one-dimensional subspaces of $T_x M$; moreover, if $x \in M \setminus S$ is on an unstable leaf, then $E^u(x)$ is tangent to the unstable leaf (and similarly for $E^s(x)$ on a stable leaf). 
		\item $E^u(x)$ and $E^s(x)$ depend continuously on $x \in M \setminus S$. 
	\end{enumerate}
\end{proposition}

We will need a stronger condition on cone invariance. For $x \in M \setminus S$ and for $0 < \alpha < 1$, define the families of cones $K^+(x)$ and $K^-(x)$ by: 
\begin{align*}
K^+(x) &= \left\{v = (\xi_1, \xi_2) \in T_x M: |\xi_2| < \alpha |\xi_1|\right\}, \\
K^-(x) &= \left\{v = (\xi_1, \xi_2) \in T_x M : |\xi_1| < \alpha |\xi_2|\right\}.
\end{align*}
In the original construction of pseudo-Anosov diffeomorphisms yielding Proposition \ref{regular-cones}, we have $\alpha = 1$. But for certain later arguments, we will require $\alpha < 1$. 

\begin{lemma}\label{strong-cones}
	There exists a $0 < \alpha_0 < 1$ such that for all $\alpha_0 < \alpha < 1$, and for all $x \in M$, 
	\begin{align*}
	Dg_x K^+(x) \subseteq K^+(g(x)) \quad \textrm{and} \quad Dg^{-1}_{g(x)} K^-(g(x)) \subseteq K^-(x).
	\end{align*}
\end{lemma}

\begin{proof}
	We prove invariance only for $K^+(x)$; the invariance of the stable cones is proven similarly by considering $g^{-1}$. Assume $x \in \tilde\Us_0$, as the result is clearly true outside of $\tilde\Us_0$. Consider the vector field (\ref{slowdown vector field}) defined on $\C$. The variational equations for (\ref{slowdown vector field}) give us
	\[
	\frac{d\zeta_1}{dt} = \log\lambda\left(\left(\Psi_p\left(u\right) + 2s_1^2 \dot{\Psi}_p\left(u\right)\right)\xi_1 + 2s_1 s_2 \dot{\Psi}_p\left(u\right)\xi_2\right)
	\]
	and
	\[
	\frac{d\zeta_2}{dt} = -\log\lambda\left(2s_1 s_2 \dot{\Psi}_p\left(u\right)\xi_1 + \left(\Psi_p\left(u\right) + 2s_2^2 \dot{\Psi}_p\left(u\right)\right)\xi_2  \right).
	\]
	where $u := s_1^2 + s_2^2$. The ``slope'' $\eta := \xi_2/\xi_1$ of a tangent vector in $\C$ changes under the flow of (\ref{slowdown vector field}) as:
	\begin{equation}\label{tangent-change}
	\frac{d\eta}{dt} = -2\log\lambda\left(\left(1+\eta^2\right)s_1s_2\dot{\Psi}_p(u) + \left(\Psi_p(u) + \left(s_1^2 + s_2^2\right)\dot\Psi_p\right)\eta\right)
	\end{equation}
	Suppose $\tilde r_1^2 \leq u \leq \tilde r_0^2$. Since $\Psi_p > 0$, and $\dot{\Psi}_p > 0$ is decreasing, we have: 
	\[
	\frac{\Psi_p(u)}{\dot{\Psi}_p(u)} \geq \frac{\Psi_p(\tilde r_1^2)}{\dot{\Psi}_p(\tilde r_1^2)} = \frac{p}{p-2} \tilde r_1^2 \geq \frac{p}{p-2}\left(\frac{\tilde r_1}{\tilde r_0}\right)^2 u.
	\]
	Meanwhile, if $0 < u < \tilde r_1^2$, we have
	\[
	\frac{\Psi_p(u)}{\dot{\Psi}_p(u)} = \frac{p}{p-2}u \geq \frac{p}{p-2}\left(\frac{\tilde r_1}{\tilde r_0}\right)^2 u.
	\]
	If $\eta > 0$, this gives us
	\begin{align*}
	\frac{d\eta}{dt} &\leq -2\log\lambda\dot{\Psi}_p(u)\left(\left(1+\eta^2\right)s_1s_2 + \left(1+\frac{p}{p-2}\left(\frac{\tilde r_1}{\tilde r_2}\right)^2\right) \left(s_1^2 + s_2^2\right)\eta\right) \\
	&= -2\log\lambda\dot\Psi_p(u)\left(\left(\left(1+\frac{p}{p-2}\left(\frac{\tilde r_1}{\tilde r_0}\right)^2\right)\eta - \frac 1 2 \left(1+\eta^2\right)\right)\left(s_1^2 + s_2^2\right) \right. \\
	&\qquad \qquad \qquad \qquad \qquad \qquad \qquad \qquad \qquad \qquad \qquad +\frac 1 2 \left(1+\eta^2\right)\left(s_1 + s_2\right)^2\Bigg) \\
	&\leq -2\log\lambda\dot{\Psi}_p(u)\psi(\eta)\left(s_1^2+s_2^2\right),
	%&\leq -2\log\lambda\dot{\Psi}_p(u)\left( \left( \frac{p}{p-2}\left(\frac{\tilde r_1}{\tilde r_2}\right)^2 \eta - \frac 1 2 \left(\eta-1\right)^2\right)\left(s_1^2 + s_2^2\right)\right).
	\end{align*}
	where $\psi(\eta) := \frac{p}{p-2}\left(\frac{\tilde r_1}{\tilde r_2}\right)^2 - \frac 1 2 (\eta-1)^2$. Since $\psi(1) > 0$, there is a $\alpha_0 \in (0,1)$ with $\psi(\eta) > 0$ for $\alpha_0 < \eta < 1$. Therefore $\frac{d\eta}{dt} < 0$ for $\alpha_0 < \eta < 1$. For $\eta < 0$, we have
	\begin{align*}
	\frac{d\eta}{dt} &= 2\log\lambda\left(\left(\Psi_p(u) + \left(s_1^2 + s_2^2\right)\dot\Psi_p(u)\right)|\eta| - s_1 s_2 \left(1+\eta^2\right) \dot{\Psi}_p(u\right) \\
	&\geq 2\log\lambda \dot{\Psi}_p(u)\left(\left(1+\frac{p}{p-2}\left(\frac{\tilde r_1}{\tilde r_0}\right)^2\right)\left(s_1^2 + s_2^2\right)|\eta| - s_1 s_2 \left(1+\eta^2\right)\right).
	\end{align*}
	A similar argument will show $\frac{d\eta}{dt} > 0$ for $-1 < \eta < -\alpha_0$. Letting $\alpha = \eta$, for $z \in \C$, we have $D(G_p)_z K^+_0(z) \subseteq K^+_0(G_p(z))$ and $D(G_p)^{-1}_{G_p(z)} K^-_0(G_p(z)) \subseteq K^-_0(z)$, where 
	\[
	K_0^+(z) = \left\{(\zeta_1, \zeta_2) \in T_z\C : |\zeta_2| < \alpha|\zeta_1|\right\},
	\]
	\[
	K^-_0(z) = \left\{(\zeta_1, \zeta_2) \in T_z\C : |\zeta_1| < \alpha|\zeta_2|\right\}.
	\]
	Note $\alpha_0$ does not depend on the distance of $z \in \C$ from 0. Applying the coordinate map $\oldphi_k^{-1} \circ \Phi_{kj}^{-1} : \{z : \mathrm{Re}(z) \geq 0\} \cap D_{\tilde a_k} \to M$, the cones $K^+(x)$ and $K^-(x)$ defined using the coordinates in (\ref{sector-coords}) for $T_xM$  satisfy the same invariance property as $K^+_0$ and $K^-_0$. This proves the lemma. 
\end{proof}
	
	\section{Main results}
	
	We begin by defining the relevant ergodic properties under consideration. Given a continuous potential function $\phi : M \to \R$, a probability measure $\mu_\phi$ on $M$ is an \emph{equilibrium measure} for $\phi$ if 
	\[
	P_g(\phi) = h_{\mu_\phi}(g) + \int_M \phi \, d\mu_\phi,
	\]
	where $h_{\mu_\phi}(g)$ is the metric entropy of $g$ with respect to $\mu_\phi$, and $P_g(\phi)$ is the topological pressure of $\phi$; that is, $P_g(\phi)$ is the supremum of $h_\mu(g) + \int_M \phi \, d\mu$ over all $g$-invariant probability measures $\mu$ on $M$.
	
	A special instance of equilibrium measures are known as SRB measures. Given a (uniformly, nonuniformly, or partially) hyperbolic function $f : M \to M$ on a Riemannian manifold $M$, an $f$-invariant Borel probability measure $\mu$ on $M$ is called an \emph{SRB measure} if $f$ admits positive Lyapunov exponents $\mu$-almost everywhere, and if the conditional measures of $\mu$ on the unstable submanifolds are absolutely continuous with respect to the Riemannian leaf volume. 
	
	Additionally, we say that $g$ has \emph{exponential decay of correlations} with respect to a measure $\mu \in \mathcal{M}(g, M)$ and a class of functions $\mathcal H$ on $M$ if there exists $\kappa \in (0,1)$ such that for any $h_1, h_2 \in \mathcal H$, 
	\[
	\left| \int h_1\left(g^n(x)\right) h_2(x) \, d\mu(x) - \int h_1(x) \, d\mu(x) \int h_2(x)\,d\mu(x)\right| \leq C\kappa^n
	\]
	for some $C = C(h_1, h_2) > 0$. Furthermore, $g$ is said to satisfy the \emph{Central Limit Theorem} (CLT) for a class $\mathcal H$ of functions if for any $h \in \mathcal H$ that is not a coboundary (ie. $h \neq h' \circ g - h'$ for any $h' \in \mathcal H$), there exists $\sigma > 0$ such that 
	\[
	\lim_{n \to \infty} \mu\bigg\{\frac{1}{\sqrt n} \sum_{i=0}^{n-1} \Big(h(g^i(x)) - \int h \, d\mu \Big) < t \bigg\} = \frac 1{\sigma \sqrt{2\pi}} \int_{-\infty}^{t} e^{-\tau^2/2\sigma^2} \, d\tau.
	\]
	The family of potential functions we consider are the \emph{geometric $t$-potentials} defined by $\phi_t(x) = -t\log\left|Dg_x\big|_{E^u(x)}\right|$. Although the unstable distribution $E^u$ does not continuously extend to the singularities, the differential $Dg_{x_0}$ is the identity at each singularity $x_0$, so $\phi_t$ continuously extends to the singularities; in particular, $\phi_t(x_0) = 0$ for each singularity $x_0$. So the geometric $t$-potential is well-defined in this setting. 
	
	Our result shows there is a $t_0 < 0$ for which every $t \in (t_0, 1)$ admits a unique equilibrium state $\mu_{\phi_t} =: \mu_t$ for the potential $\phi_t : M \to \R$. When $t=0$, $\phi_0 \equiv 0$, so the equilibrium measure $\mu_0$ satisfies $P_g(0) = h_{\mu_0}(g)$, and so $\mu_0$ is the unique measure of maximal entropy for $g$. 
	
	We now state our main result. 
	
	\begin{theorem}\label{main-theorem}
		Consider a pseudo-Anosov diffeomoprhism $g : M \to M$ on a compact Riemannian manifold $M$. The following statements hold: 
		\begin{enumerate}
			\item Given any $t_0 < 0$, we may take $r_0 > 0$ in the construction of $g$ so that for any $t \in (t_0, 1)$, there is a unique equilibrium measure $\mu_t$ associated to $\phi_t$. This equilibrium measure has exponential decay of correlations and satisfies the Central Limit Theorem with respect to a class of functions containing all H\"older continuous functions on $M$, and is Bernoulli. Additionally, the pressure function $t \mapsto P_g(\phi_t)$ is real analytic in the open interval $(t_0, 1)$.% For $t=0$, this corresponds to a unique measure of maximal entropy. 
			\item For $t=1$, there are two classes of equilibrium measures associated to $\phi_1$: convex combinations of Dirac measures concentrated at the singularities, and a unique invariant SRB measure $\mu$.
			\item For $t > 1$, the equilibrium measures associated to $\phi_t$ are precisely the convex combinations of Dirac measures concentrated at the singularities.  
			%\item The lower bound $t_0$ on the existence and uniqueness of equilibrium states $\mu_t$ may be made to approach $-\infty$ as $r_0 \to 0$ in the construction of $g$. 
		\end{enumerate}
	\end{theorem}
	
	\begin{remark}
		Uniqueness of the measure $\mu_t$ for $t \in (t_0, 1)$ implies this measure is ergodic, but in fact, Theorem \ref{main-theorem} gives us that this measure is Bernoulli. 
	\end{remark}
	
	\begin{remark}
		Taking $t=0$, this theorem shows that the dynamical system $(M, g)$ admits a unique measure of maximal entropy that is Bernoulli, has exponential decay of correlations, and satisfies the Central Limit Theorem. 
	\end{remark}
	
	\begin{remark}
		Although we know $t \mapsto P_g(\phi_t)$ is real analytic in $(t_0, 1)$, we do not know about the behavior of $P_g(\phi_t)$ for $t \leq t_0$. In particular, it is not known whether $(M, g, \phi_t)$ admits a phase transition at $t=t_0$.\footnote{For the Katok map, it is shown in \cite{Wang} that for sufficiently small values of the parameters $\alpha>0$ and $r>0$, the Katok map has a unique equilibrium measure $\mu_t$ corresponding to the geometric potential $\varphi_t$ for all values of $t<1$.} 
	\end{remark}
	
	\section{Dynamics near singularities}
	
	In this section, we discuss the dynamical properties of pseudo-Anosov diffeomorphisms, considering both their global behavior as well as their behavior near singularities. The thermodynamic constructions we will develop in Sections 6 and 7 require bounds on how quickly nearby orbits diverge from each other. For this reason, the estimates and inequalities collected in this section will become important tools to examine how nearby orbits behave in neighborhoods of the singularities.
	
	Several of the technical calculations made here are similar to the calculations performed for the Katok map in \cite{PSZ17}. However, they are carried out here for the reader's convenience, as well as the fact that the slowdown function in the Katok map uses different constants depending on the radius of the slowed-down neighborhood (by contrast, our slowdown function depends not on the radius of the slowdown, but on the number of prongs of the singularity). 
	
	Our first two technical estimates concern how long an orbit remains in a neighborhood of a singularity. Recall our definitions $\tilde r_j = (2/p)r^{p/2}_j$ for $j=0,1$. In particular, $\tilde r_0$ and $\tilde r_1$ depend on $p$, and thus depend on $k$ for $k=1,\ldots,m$. 
	
	\begin{lemma}\label{annular bound in C}
		There exists a $T_p > 0$, depending on $p$, $\lambda$, $r_0$, and $r_1$, so that for any solution $s(t)$ of \emph{(\ref{slowdown vector field})} with $s(0) \in D_{\tilde r_0}$, 
		\[
		\max\left\{t > 0 : s(t) \in D_{\tilde r_0} \setminus D_{\tilde r_1} \right\} < T_p.
		\]
	\end{lemma}
	
	\begin{proof} The value $s_1 s_2$ is invariant under the flow. If $s_1 s_2 \geq \frac 1 2 \tilde r_1^2$, then when $s_1 = s_2$, the minimum value of $s_1^2 + s_2^2$ is $\geq \tilde r_1^2$, and the trajectory never enters $D_{\tilde r_1}$. If $s_1 s_2 < \frac 1 2 \tilde r_1^2$, the trajectory either will enter $D_{\tilde r_1}$ or has already entered $D_{\tilde r_1}$ and is on its way out of $D_{\tilde r_0}$. 
		
		\textbf{Case 1:} $s_1 s_2 \geq \frac 1 2 \tilde r_1^2$. Since $\tilde r_0^2 \geq s_1^2 + s_2^2 \geq s_2^2$, we have $\frac 1 4 \tilde r_1^4 \leq s_1^2 s_2^2 \leq s^2_1 \tilde r_0^2$, so $s_1^2 \geq \tilde r_1^4 / 4\tilde r_0^2$. So, since $\Psi_p$ is an increasing function, 
		\begin{align*}
		\frac d{dt} \left(s_1^2\right) &= 2s_1^2 \Psi_p\left(s_1^2 + s_2^2\right) \log\lambda \geq \frac{\tilde r_1^4}{2\tilde r_0^2} \Psi_p\left(\tilde r_1^2\right) \log\lambda.
		\end{align*}
		It follows that the time $T$ it takes for $s_1^2$ to reach $\tilde r_0^2$ from $s_1^2(0) \geq \tilde r_1^4/4\tilde r_0^2$ satisfies
		\[
		T \leq \frac{\tilde r_0^2 - \frac{\tilde r_1^4}{4\tilde r_0^2}}{\frac{\tilde r_1^4}{2\tilde r_0^2}\Psi_p\left(\tilde r_1^2\right)\log\lambda} = \frac{4r_0^{2p} - r_1^{2p}}{2r_1^{3p-2}\log\lambda}. 
		\]
		
		\textbf{Case 2:} $s_1 s_2 < \frac 1 2 \tilde r_1^2$. Assume that $s_1 < s_2$, ensuring that the trajectory will enter $D_{\tilde r_1}$. If we can prove there is a uniform time bound $T$ before which this happens, then by symmetry of the vector field, the same $T$ is an upper bound for the time it takes this trajectory to exit $D_{\tilde r_0}$ when $s_1 > s_2$. 
		
		We will in fact establish a bound on how long it takes $s_2^2$ to decrease from $s_2^2(0)$ to $\frac 1 2 \tilde r_1^2$ when $s_1 < s_2$. For then, because $s_1 s_2 < \frac 1 2 \tilde r_1^2$, by the time $s_2^2 = \frac 1 2 \tilde r_1^2$, the trajectory will already have entered $D_{\tilde r_1}$. So, $s_2^2 \geq \frac 1 2 \tilde r_1^2$, and since in this case $s_1^2 + s_2^2 \geq \frac 1 2 \tilde r_1^2$, we have
		\[
		\frac d{dt}\left(s_2^2\right) = -2s_2^2 \Psi_p\left(s_1^2 + s_2^2\right)\log\lambda \leq -\tilde r_1^2 \Psi_p \left(\textstyle\frac 1 2 \tilde r_1^2\right) \log\lambda.
		\]
		It follows that the time $T$ it takes for $s_2^2$ to reach $\frac 1 2 \tilde r_1^2$ from $s_2^2(0) \leq \tilde r_0^2$ satisfies
		\[
		T \leq \frac{\tilde r_0^2 - \frac 1 2 \tilde r_1^2}{\tilde r_1^2 \Psi_p\left(\frac 1 2 \tilde r_1^2\right) \log\lambda} = 2^{(p-2)/2} \frac{2r_0^p - r_1^p}{2r_1^{2p-2} \log\lambda}.
		\] 
		%Let $u = s_1^2 + s_2^2$. For $s_1 \leq s_2$, we have
		%\[
		%\frac{du}{dt} = 2\Psi_p\left(s_1^2 + s_2^2\right) \log\lambda\left(s_1^2 - s_2^2\right) = -2\Psi_p \left(s_1^2 + s_2^2\right)\log\lambda\left(u^2 - 4s_1^2s_2^2\right)^{1/2}.
		%\]
		%Meanwhile for $s_2 \leq s_1$, we get:
		%\[
		%\frac{du}{dt} = 2\Psi_p \left(s_1^2 + s_2^2\right)\log\lambda\left(u^2 - 4s_1^2s_2^2\right)^{1/2}.
		%\] 
		%Working with the assumption $s_1 \leq s_2$, we get... [FILL IN LATER]
	\end{proof} 
	
	\begin{lemma}\label{annular bound in M}
		There exists a $T \in \Z$, depending on $r_0$ and $\lambda$, so that for any $x \in \tilde \Us_0 := \union_{k=1}^m \oldphi_k^{-1}\left(D_{\tilde r_0}\right) \subset M$, we have 
		\[
		\max\left\{ N > 0 : g^n(x) \in \union_{k=1}^m \oldphi_k^{-1}\left(D_{\tilde r_0} \setminus D_{\tilde r_1}\right) \: \textrm{for all } n = 0, \ldots N\right\} \leq T.
		\]
	\end{lemma}
	\begin{proof}
		This follows from Lemma \ref{annular bound in C} after taking $T = \max\{T_{p(k)} : k = 1, \ldots, m\}$. 
	\end{proof}
	
	Next, we will establish bounds on how quickly nearby points will diverge while remaining near the singularities. The main lemma that demonstrates this bound is Lemma \ref{spread-in-slowdown lemma}. 
	
	\begin{lemma}\label{dij-estimate}
		For $i,j = 1,2$ define the functions $d_{ij} : D_{\tilde r_1} \to \R$ by
		\[
		d_{ij}(s_1, s_2) = \frac{\del^2}{\del s_i \del s_j} \left( s_2 \Psi_p\left(s_1^2 + s_2^2\right)\right).
		\]
		Then,
		\[
		\max_{i,j = 1,2} |d_{ij}(s_1, s_2)| \leq \frac{6p-12}{p}\left(\frac p 2\right)^{(2p-4)/p} \left(s_1^2 + s_2^2\right)^{(p-4)/2p}.
		\]
	\end{lemma}
	
	\begin{proof}
		Recall that for $u \leq \tilde r_1^2$, we have $\Psi_p(u) = (p/2)^{(2p-4)/p} u^{(p-2)/p}$. So, 
		\begin{align*}
		\frac{\del}{\del s_1}\left(s_2 \Psi_p\left(s_1^2 + s_2^2\right)\right) &= \frac{2p-4}{p}\left(\frac p 2\right)^{(2p-4)/p} s_1 s_2\left(s_1^2 + s_2^2\right)^{-2/p}, \quad \textrm{and} \\
		\frac{\del}{\del s_2}\left(s_2 \Psi_p\left(s_1^2 + s_2^2\right)\right) &= \frac{2p-4}{p}\left(\frac p 2\right)^{(2p-4)/p} s_2^2\left(s_1^2 + s_2^2\right)^{-2/p}\\
		&\blank\blank + \left(\frac p 2\right)^{(2p-4)/p} \left(s_1^2 + s_2^2\right)^{(p-2)/p}.
		\end{align*}
		Note $|s_1|^2 \leq \sqrt{s_1^2 + s_2^2}$, and since $p \geq 3$, 
		\[
		-2 \leq -\frac{4s_1^2}{p\left(s_1^2 + s_2^2\right)} \leq 0.
		\]
		Therefore, for all $(s_1, s_2) \in D_{\tilde r_1}$, 
		\begin{align*}
		\left|d_{11}\left(s_1, s_2\right)\right| &= \frac{2p-4}{p}\left(\frac p 2\right)^{(2p-4)/p} \left| \frac{\del}{\del s_1} s_1 s_2 \left(s_1^2 + s_2^2\right)^{-2/p}\right| \\
		&= \frac{2p-4}{p}\left(\frac p 2\right)^{(2p-4)/p} \left| s_2 \left(s_1^2 + s_2^2\right)^{-2/p} - \frac 4 p s_1^2 s_2 \left(s_1^2 + s_2^2\right)^{-(p+2)/p}\right| \\
		&= \frac{2p-4}{p}\left(\frac p 2\right)^{(2p-4)/p} |s_2| \left(s_1^2 + s_2^2\right)^{-2/p}\left| 1 - \frac{4s_1^2}{p\left(s_1^2 + s_2^2\right)}\right| \\
		&\leq \frac{2p-4}{p}\left(\frac p 2\right)^{(2p-4)/p}\left(s_1^2 + s_2^2\right)^{(p-4)/2p}.
		\end{align*}
		A similar argument applies for $d_{12} = d_{21}$ and for $d_{22}$, though in $d_{22}$ we use the estimate $-2 \leq 4s_1^2/3p\left(s_1^2+s_2^2\right)$ instead: 
		\begin{align*}
		\left|d_{12}\left(s_1, s_2\right)\right| &= \frac{2p-4}{p}\left(\frac p 2\right)^{(2p-4)/p} |s_1| \left(s_1^2 + s_2^2\right)^{-2/p}\left| 1 - \frac{4s_2^2}{p\left(s_1^2 + s_2^2\right)}\right| \\ 
		&\leq \frac{2p-4}{p}\left(\frac p 2\right)^{(2p-4)/p}\left(s_1^2 + s_2^2\right)^{(p-4)/2p}, 
		\end{align*}
		\begin{align*}
		\left|d_{22}\left(s_1,s_2\right)\right| &= \frac{6p-12}{p} \left(\frac p 2\right)^{(2p-4)/p}\left|s_2\right| \left(s_1^2 + s_2^2\right)^{-2/p} \left| 1 - \frac{4s_2^2}{3p\left(s_1^2 + s_2^2\right)}\right| \\
		&\leq \frac{6p-12}{p}\left(\frac p 2\right)^{(2p-4)/p}\left(s_1^2 + s_2^2\right)^{(p-4)/2p}.
		\end{align*}
	\end{proof}
	
	Let $s(t) = \big( s_1(t), \, s_2(t)\big)$ be a solution to (\ref{slowdown vector field}), and assume $s(t)$ is defined in the unique interval $[0,T]$ for which $G_p^{-1}(s(0)), G_p(s(T)) \not\in D_{\tilde r_1}$ and $s(t) \in \overline D_{\tilde r_1}$ for $0 \leq t \leq T$. In particular, this means $s(0), s(T) \in \partial D_{\tilde r_1}$. (Recall $G_p$ is the time-1 map of the vector field (\ref{slowdown vector field}).) Further denote $T_1 = T/2$, so that if $s_1(t) > 0$ and $s_2(t) > 0$ for $t \in [0,T]$, we have $s_1(t) \leq s_2(t)$ for $t \in [0,T_1]$ and $s_1(t) \geq s_1(t)$ for $t \in [T_1, T]$. 
	
	\begin{lemma}\label{s1 and s2 approx}
		Given a solution $s(t)$ to \emph{(\ref{slowdown vector field})}, and $T$ and $T_1$ defined above, we have the following inequalities: 
		\begin{enumerate}[label=\emph{(\alph*)}]
			\item $|s_1(t)| \leq |s_1(b)| \left(1+C_0 s_1(b)^{(2p-4)/p}(b-t)\right)^{-p/(2p-4)}$, $\quad 0 \leq t \leq b \leq T$;
			\item $|s_2(t)| \leq |s_2(a)| \left(1+C_0s_2(a)^{(2p-4)/p} (t-a)\right)^{-p/(2p-4)}$, $\quad 0 \leq a \leq t \leq T$;
			\item $|s_2(t)| \geq |s_2(a)|\left(1+2^{(p-2)/p} C_0 s_2(a)^{(2p-4)/p}(t-a)\right)^{-p/(2p-4)}$,\\ $0 \leq a \leq t \leq T_1$; 
			\item $|s_1(t)| \geq |s_1(b)| \left(1+2^{(p-2)/p}C_0 s_1(b)^{(2p-4)/p}(b-t)\right)^{-p/(2p-4)},$ \\ $T_1 \leq t \leq b \leq T$;
			%\item $|s_2(t)| \leq |s_2(a)|\left( 1 + C_0s_2(a)^{(2p-4)/p}(t-a)\right)^{-p/(2p-4)}$, $0 \leq a \leq t \leq T$; 
			%		\item $|s_1(t)| \geq |s_1(a)|\left(1-C_0 s_1(a)^{(2p-4)/p}(t-a)\right)^{-p/(2p-4)}$, $0 \leq a \leq t \leq T$; 
			%		\item $|s_1(t)| \leq |s_1(T_1)|\left(1-2^{(p-2)/p}C_0 s_1(T_1)^{(2p-4)/p}(t-T_1)\right)^{-p/(2p-4)}$, \\$T_1 \leq t \leq T$;
			%\item $|s_1(t)| \leq |s_1(b)|\left(1+C_0s_1(b)^{(2p-4)/p}(b-t)\right)^{-p/(2p-4)}$, $0 \leq t \leq b \leq T$;
		\end{enumerate}
		where $C_0 = \frac{2p-4}{p}\left(\frac p 2\right)^{(2p-4)/p}\log\lambda$. %In particular, 
		%	\[
		%	T \leq \frac{2^{2/p} ps_1(T_1)^{-(2p-4)/p}}{(2p-4)\log\lambda} \left(\frac 2 p \right)^{(2p-4)/p} = \frac{2^{(2p-2)/p} s_1(T_1)^{-(2p-4)/p}}{p^{(p-4)/p}(2p-4)\log\lambda}.
		%	\]
	\end{lemma}
	
	\begin{proof}
		By symmetry, we may assume $s_1(t) > 0$ and $s_2(t) > 0$ for $t \in [0,T]$. Then using the facts that $s_1^2 + s_2^2 \geq s_i^2$ for $i=1,2$, and that $\Psi_p(u) = (p/2)^{(2p-4)/p}u^{(p-2)/p}$ for $0 \leq u \leq \tilde r_1^2$, (\ref{slowdown vector field}) implies
		\begin{align*}
		\frac{d}{dt} s_1(t) &\geq \left( \frac p 2 \right)^{(2p-4)/p} s_1(t)^{(3p-4)/p}\log\lambda, \quad \textrm{and} \\
		\frac{d}{dt} s_2(t) &\leq -\left(\frac p 2 \right)^{(2p-4)/p} s_2(t)^{(3p-4)/p} \log\lambda. 
		\end{align*}
		In particular, this gives us 
		\begin{align*}
		s_1(t)^{-(3p-4)/p} \frac{d}{dt}s_1(t) &\geq \left(\frac p 2 \right)^{(2p-4)/p} \log\lambda, \quad \textrm{and} \\
		s_2(t)^{-(3p-4)/p} \frac{d}{dt}s_2(t) &\leq -\left(\frac p 2 \right)^{(2p-4)/p} \log\lambda.
		\end{align*}
		Integrating these expressions between $a$ and $b$, where $0 \leq a \leq b \leq T$, we get:
		\begin{align*}
		s_2(b)^{-(2p-4)/p} - s_2(a)^{-(2p-4)/p} &\geq C_0(b-a), \quad \textrm{and} \\
		s_1(b)^{-(2p-4)/p} - s_1(a)^{-(2p-4)/p} &\leq -C_0(b-a),
		\end{align*}
		where $C_0 = \frac{2p-4}{p}\left(\frac p 2\right)^{(2p-4)/p}\log\lambda$. From assuming that $s_i(t) > 0$, $i=1,2$, we get inequalities (a) and (b).
		
		Using the fact that $s_1(t) \leq s_2(t)$ for $0 \leq t \leq T_1 = \frac 1 2 T$ and $s_1(t) \geq s_2(t)$ for $T_1 \leq t \leq T$, we get:
		\begin{align*}
		s_1(t)^2 + s_2(t)^2 &\leq 2s_2(t)^2, \quad 0 \leq t \leq T_1; \\
		s_1(t)^2 + s_2(t)^2 &\leq 2s_1(t)^2, \quad T_1 \leq t \leq T.
		\end{align*}
		Once again, applying (\ref{slowdown vector field}) yields
		\begin{align*}
		\frac d{dt} s_1(t) &\leq 2^{(p-2)/p} \left( \frac p 2\right)^{(2p-4)/p}s_1(t)^{(3p-4)/p}\log\lambda, \quad T_1 \leq t \leq T, \\
		\frac d{dt} s_2(t) &\geq -2^{(p-2)/p} \left( \frac p 2\right)^{(2p-4)/p}s_2(t)^{(3p-4)/p}\log\lambda, \quad 0 \leq T_1 \leq T.
		\end{align*}
		Using the same integration strategy from $a$ to $b$ as before gives us 
		\begin{align*}
		s_1(b)^{-(2p-4)/p}-s_1(t)^{-(2p-4)/p} &\geq -2^{(p-2)/p}C_0(b-t), \quad T_1 \leq t \leq b \leq T; \\
		s_2(t)^{-(2p-4)/p}-s_2(a)^{-(2p-4)/p} &\leq 2^{(p-2)/p}C_0(t-a), \quad 0 \leq a \leq t \leq T_1.
		\end{align*}
		This gives us inequalities (c) and (d). %The final inequality follows from (d) after taking $t=T$, since $T-T_1 = T/2$. 
	\end{proof}
	
	Now suppose $\tilde s(t) = \big( \tilde s_1(t), \, \tilde s_2(t)\big)$ is another solution of (\ref{slowdown vector field}) defined for $t \in [0,T]$. We will need an upper and lower bound for $\Delta s(t) := \tilde s(t) - s(t)$. Let $\Delta s_j(t) = \tilde s_j(t) - s_j(t)$, $j=1,2$. 
	
	\begin{lemma}\label{spread-in-slowdown lemma}
		Suppose $s_1(t) \neq 0 \neq s_2(t)$ for $t \in [0,T]$ and that $\Delta s_2(t) > 0$ for $t \in [0,T]$. Suppose further that $0 < \alpha < 1$ satisfies
		\begin{enumerate}[label=\emph{(\arabic*)}]
			\item $\left| \Delta s_1(t)\right| \leq \alpha \Delta s_2(t)$ for $t \in [0,T]$; 
			\item $\left| \frac{\Delta s_2(0)}{s_2(0)}\right| \leq \frac{1-\alpha}{72}$. 
		\end{enumerate}
		Then, 
		\begin{align*}
		\Delta s_2(t) &\leq \frac{\Delta s_2(0)}{s_2(0)} s_2(t) \left(1 + 2^{(p-2)/p}C_0s_2(0)^{(2p-4)/p}t\right)^{-\beta}, \qquad 0 \leq t \leq T_1, \\
		\Delta s_2(t) &\leq \frac{\Delta s_2(T_1)}{s_1(T_1)} s_1(t) \left( \frac{1+2^{(p-2)/p}C_0 s_1(b)^{(2p-4)/p}(b-t)}{1+2^{(p-2)/p}C_0 s_1(b)^{(2p-4)/p}(b-T_1)} \right)^{\beta}, \qquad T_1 \leq t \leq b \leq T, %\\
		%	\Delta s_2(t) &\leq \frac{\Delta s_2(T_1)}{s_1(T_1)} s_1(T) \left(1+C_0 s_1(T)^{(2p-4)/p}(T-t)\right)^{-p/(2p-4)}, \qquad T_1 \leq t \leq T,
		\end{align*}
		where $\beta = 2^{-(3p-2)/p}(1-\alpha)$,
		and $C_0$ is the constant from Lemma \ref{s1 and s2 approx}. Furthermore, for $0 \leq a \leq T_1 \leq b \leq T$,
		\begin{equation}\label{deviation bound}
		\norm{\Delta s(b)} \leq \sqrt{1+\alpha^2} \frac{s_1(b)}{s_2(a)} \norm{\Delta s(a)}. 
		\end{equation}
	\end{lemma}
	
	\begin{proof}
		Assume $s_j(t) > 0$ for $j=1,2$; the other cases follow by symmetry. Further denote $u = s_1^2 + s_2^2$ and $\tilde u = \tilde s_1^2 + \tilde s_2^2$. Applying equation (\ref{slowdown vector field}) to the second Lagrange remainder of the function $(s_1, s_2) \mapsto s_2 \Psi_p\left(s_1^2 + s_2^2\right)$ centered at the point $(s_1, s_2)$, we get:
		\begin{align*}
		\frac{d}{dt} \Delta s_2 &= -\log\lambda\left( \tilde s_2 \Psi_p(\tilde u) - s_2 \Psi_p(u)\right) \\
		&= -\log\lambda \Bigg( \frac{\del}{\del s_1} \Big( s_2 \Psi_p(u)\Big) \Delta s_1 + \frac{\del}{\del s_2}\Big( s_2 \Psi_p(u)\Big) \Delta s_2 + \frac 1 2 \sum_{j,k = 1,2} d_{jk} \left(\xi_1, \xi_2\right) \Delta s_j \Delta s_k\Bigg) \\
		&= -\log\lambda \Bigg(2s_1 s_2 \dot\Psi_p(u)\Delta s_1 + \left( \Psi_p(u) + 2s_2^2 \dot\Psi_p(u)\right)\Delta s_2 \\
		&\qquad \qquad \qquad + \frac 1 2 \sum_{j,k = 1,2} d_{jk} \left(\xi_1, \xi_2\right) \Delta s_j \Delta s_k\Bigg),
		\end{align*}
		where $d_{jk}$ are as in Lemma (\ref{dij-estimate}) and $\xi = (\xi_1, \xi_2) \in D_{\tilde r_1}$ is such that $\xi_j$ lies between $s_j$ and $\tilde s_j$ for $j=1,2$. %Observe that
		%	\[
		%	\frac{\del}{\del s_1}\left(s_2 \Psi_p(u)\right) = 2s_1 s_2 \dot{\Psi}_p(u), \quad \frac{\del}{\del s_2}\left(s_2\Psi_p(u)\right) = \Psi_p(u) + 2s_2^2 \dot\Psi_p(u).
		%	\]
		It follows that 
		\begin{align*}
		\frac d{dt}\left(\frac{\Delta s_2}{s_2}\right) &= \frac 1{s_2} \frac d{dt} \Delta s_2 - \frac 1{s_2^2} \dot s_2 \Delta s_2 \\ 
		&= -\log\lambda \left(2s_1\dot\Psi_p(u) \Delta s_1 + \frac 1{s_2} \Psi_p(u) \Delta s_2 + 2s_2 \dot\Psi_p(u)\Delta s_2 \right) \\ 
		&\qquad -\frac{\log\lambda}{2} \sum_{j,k = 1,2} d_{jk} \left(\xi_1, \xi_2\right)\frac{ \Delta s_j \Delta s_k}{s_2}+\log\lambda \frac{1}{s_2} \Psi_p(u)\Delta s_2 \\
		&= -\frac{(2p-4)\log\lambda}{p}\left(\frac p 2\right)^{(2p-4)/p} u^{-2/p}\left(s_1\Delta s_1 + s_2 \Delta s_2\right) \\
		&\qquad -\frac{\log\lambda}{2} \sum_{j,k = 1,2} d_{jk} \left(\xi_1, \xi_2\right) \frac{\Delta s_j \Delta s_k}{s_2}. 
		\end{align*}
		Suppose $0 \leq t \leq T_1$, so that $0 < s_1(t) \leq s_2(t)$. Since $|\Delta s_1(t)| \leq \alpha\Delta s_2(t)$ by assumption, we get:
		\[
		s_1\Delta s_1 + s_2 \Delta s_2 \geq \left(-s_1\alpha + s_2\right) \Delta s_2 \geq (1-\alpha)s_2\Delta s_2.
		\]
		Lemma \ref{dij-estimate}  implies
		\begin{equation}\label{dij-sum-est}
		\sum_{j,k} d_{jk}\left(\xi_1, \xi_2\right) \Delta s_j \Delta s_k \geq -\frac{24p-48}{p}\left(\frac p 2\right)^{(2p-4)/p}\left(\xi_1^2 + \xi_2^2\right)^{(p-4)/2p}\left(\Delta s_2\right)^2.
		\end{equation}
		It follows from the above two inequalities that 
		\begin{align*}
		\frac d{dt}\left(\frac{\Delta s_2}{s_2}\right) &\leq -(1-\alpha)\frac{(2p-4)\log\lambda}{p}\left(\frac p 2\right)^{(2p-4)/p} \left(s_1^2 + s_2^2\right)^{-2/p} s_2\Delta s_2 \\
		&\qquad  +\frac{(12p-24)\log\lambda}{p} \left(\frac p 2\right)^{(2p-4)/p}\left(\xi_1^2 + \xi_2^2\right)^{(p-4)/2p}\frac{\left(\Delta s_2\right)^2}{s_2}.
		\end{align*}
		Since $s_1(t) \leq s_2(t)$ for $0 \leq t \leq T_1$, we have $s_2^2 \leq s_1^2 + s_2^2 \leq 2s_2^2$. Therefore, 
		\begin{align*}
		\frac d{dt}\left(\frac{\Delta s_2}{s_2}\right) &\leq -(1-\alpha)\frac{(2p-4)\log\lambda}{p}\left(\frac p 2\right)^{(2p-4)/p} \left(s_1^2 + s_2^2\right)^{(p-2)/p}\frac{s_2^2}{s_1^2 + s_2^2} \frac{\Delta s_2}{s_2} \\
		&\quad+\frac{(12p-24)\log\lambda}{p} \left(\frac p 2\right)^{(2p-4)/p}s_2^{(2p-4)/p}\left(\frac{\xi_1^2 + \xi_2^2}{s_2^2}\right)^{(p-4)/2p}\left(\frac{\Delta s_2}{s_2}\right)^2\\
		&\leq -(1-\alpha)\frac{(p-2)\log\lambda}{p}\left(\frac{ps_2} 2\right)^{(2p-4)/p}\frac{\Delta s_2}{s_2} \\
		&\quad+\frac{(12p-24)\log\lambda}{p} \left(\frac{ps_2} 2\right)^{(2p-4)/p}\left(\frac{\xi_1^2 + \xi_2^2}{s_2^2}\right)^{(p-4)/2p}\left(\frac{\Delta s_2}{s_2}\right)^2.
		\end{align*}
		Denoting $\kappa = \kappa(t) = \frac{\Delta s_2}{s_2}(t)$, we summarize:
		\begin{align*}
		\frac{d\kappa}{dt} &\leq  -(1-\alpha)\frac{(p-2)\log\lambda}{p}\left(\frac{ps_2} 2\right)^{(2p-4)/p} \kappa \\
		&\qquad+ \frac{(12p-24)\log\lambda}{p} \left(\frac{ps_2} 2\right)^{(2p-4)/p}\left(\frac{\xi_1^2 + \xi_2^2}{s_2^2}\right)^{(p-4)/2p} \kappa^2 \\
		&= -\frac{(p-2)\log\lambda}{p}\left(\frac{ps_2}{2}\right)^{(2p-4)/p}\kappa\left(1-\alpha-12\left(\frac{\xi_1^2+\xi_2^2}{s_2^2}\right)^{(p-4)/2p}\kappa\right) \numberthis \label{dkappa-approx} %\log\lambda\left(\frac{ps_2}{2}\right)^{(2p-4)/p}\kappa\left(\frac{1-\mu}{p-2}+\frac{12p-24}p\left(\frac{\xi_1^2 + \xi_2^2}{s_2^2}\right)^{(p-4)/2p}\kappa\right). 
		\end{align*}
		Note $0 < s_2 \leq \xi_2 \leq \tilde s_2 = s_2 + \Delta s_2$, and $\xi_1 \leq s_1 + |\Delta s_1| \leq s_2 + \alpha\Delta s_2$. Therefore, 
		\begin{equation}\label{xi-est-preT}
		1 \leq \frac{\xi_2^2}{s_2^2} \leq \frac{\xi_1^2 + \xi_2^2}{s_2^2} \leq \frac{\left(s_2+\alpha\Delta s_2\right)^2 + \left(s_2 + \Delta s_2\right)^2}{s_2^2} = (1+\alpha\kappa)^2 + (1+\kappa)^2 < 2(1+\kappa)^2.
		\end{equation}
		It follows that 
		\[
		\left( \frac{\xi_1^2 + \xi_2^2}{s_2^2}\right)^{(p-4)/2p} \leq \begin{cases}
		1 & \textrm{if } p=3,4; \\
		\left(2(1+\kappa)^2\right)^{(p-4)/2p} & \textrm{if } p\geq 5.
		\end{cases}
		\]
		Using Assumption (2), we observe that
		\[
		1-\alpha-12\left(\frac{\xi_1^2+\xi_2^2}{s_2^2}\right)^{(p-4)/2p}\kappa(0) \geq \frac{1-\alpha}2.
		\]
		Equation (\ref{dkappa-approx}) now implies 
		\[
		\frac{d\kappa}{dt}\bigg|_{t=0} \leq -\frac{(1-\alpha)(p-2)\log\lambda}{2p}\left(\frac{ps_2(0)}{2}\right)^{(2p-4)/p}\kappa(0) < 0.
		\]
		So $\kappa(t)$ satisfies 
		\begin{equation}\label{small-kappa}
		0 < \kappa(t) < \frac{1-\alpha}{72}
		\end{equation}
		for $0 \leq t < \delta$ for a small number $\delta > 0$. The same arguments as before now imply 
		\begin{equation}\label{dkappa-approx-2}
		\frac{d\kappa}{dt} \leq -\frac{(1-\alpha)(p-2)\log\lambda}{2p}\left(\frac{ps_2(t)}{2}\right)^{(2p-4)/p}\kappa(t) < 0
		\end{equation}
		for $0 \leq t < \delta$. Since $\kappa$ and $s_2$ are continuous and positive on $[0,T_1]$, the estimates (\ref{small-kappa}) and (\ref{dkappa-approx-2}) apply for $0 \leq t \leq T_1$. Applying Gr\"onwall's inequality to (\ref{dkappa-approx-2}) gives us for $0 \leq t \leq T_1$:
		\begin{equation}\label{Gronwell-1}
		\kappa(t) \leq \kappa(0) \exp\left( -\frac{(1-\alpha)(p-2)\log\lambda}{2p}\left(\frac{p}{2}\right)^{(2p-4)/p} \int_0^t s_2(\tau)^{(2p-4)/p}\,d\tau\right).
		\end{equation}
		Applying the third inequality in Lemma \ref{s1 and s2 approx} to this integral gives us:
		\begin{align*}
		\int_0^t s_2(\tau)^{(2p-4)/p}\,d\tau &\geq \int_0^t s_2(0)^{(2p-4)/p}\left(1+2^{(p-2)/p} C_0 s_2(0)^{(2p-4)/p}\tau\right)^{-1}\,d\tau \\
		&= \frac{1}{2^{(p-2)/p}C_0}\log \left( 1+2^{(p-2)/p} C_0 s_2(0)^{(2p-4)/p}t\right).
		%	&= \frac{p\log\lambda}{2^{(3p-2)/p}(p-2)}\left( \frac 2 p \right)^{(2p-4)/p}\log\left( 1+2^{(p-2)/p} C_0 s_2(0)^{(2p-4)/p}t\right) 
		\end{align*}%\\ 
		%	&\leq \kappa(0) \exp\bigg( -\frac{(1-\mu)(p-2)\log\lambda}{2p}\left(\frac{p}{2}\right)^{(2p-4)/p}\\
		%	&\qquad \qquad \qquad \times \int_0^t s_2(0)^{(2p-4)/p}\left(1+2^{(p-2)/p} C_0 s_2(a)^{(2p-4)/p}(t-a)\right)^{-1}\,d\tau\bigg) 
		%	\end{align*}
		Recalling that $C_0 = \frac{2p-4}{p}\left(\frac p 2\right)^{(2p-4)/p}\log\lambda$, (\ref{Gronwell-1}) now becomes: 
		\begin{align*}
		\kappa(t) &\leq \kappa(0) \exp\left( -\frac{(1-\alpha)}{2^{(3p-2)/p}} \log\left( 1+2^{(p-2)/p} C_0 s_2(0)^{(2p-4)/p}t\right)\right) \\
		&= \kappa(0) \left( 1+2^{(p-2)/p} C_0 s_2(0)^{(2p-4)/p}t\right)^{-\beta}, \numberthis\label{kappa-upper}
		\end{align*}
		giving us the first inequality of the lemma. 
		
		To prove the second inequality, arguing as before for $T_1 \leq t \leq T$, we get:
		\begin{align*}
		\frac d{dt}\Delta s_2 &= -\log\lambda \Bigg( \frac{\del}{\del s_1} \Big( s_2 \Psi_p(u)\Big) \Delta s_1 + \frac{\del}{\del s_2}\Big( s_2 \Psi_p(u)\Big) \Delta s_2 \\
		&\qquad \qquad \qquad + \frac 1 2 \sum_{j,k = 1,2} d_{jk} \left(\xi_1, \xi_2\right) \Delta s_j \Delta s_k\Bigg) 
		\end{align*}
		for $\xi = (\xi_1, \xi_2)$ satisfying $\min\left\{ s_j, \tilde s_j \right\} \leq \xi_j \leq \max \left\{s_j, \tilde s_j \right\}$. Thus, using assumption (1) and positivity of $s_i$, $\dot \Psi_p$, and $\Delta s_2$, 
		\begin{align*}
		\frac{d}{dt} \left( \frac{\Delta s_2}{s_1}\right) &= \frac 1{s_1} \frac d{dt} \Delta s_2 - \frac{1}{s_1^2}\dot s_1 \Delta s_2 \\
		&= -\log\lambda\left( 2s_1s_2 \dot\Psi_p(u) \frac{\Delta s_1}{s_1} + \left(2s_2^2 \dot\Psi_p(u)+ \Psi_p(u)\right) \frac{\Delta s_2}{s_1}\right) \\
		&\qquad -\frac 1 2 \log\lambda \sum_{j,k = 1,2} d_{jk} \left(\xi_1, \xi_2\right) \frac{\Delta s_j \Delta s_k}{s_1} - \log\lambda\Psi_p(u)\frac{\Delta s_2}{s_1} \\
		&\leq -2\log\lambda \left( \Psi_p(u) - \alpha s_1 s_2 \dot \Psi_p(u) + s_2^2 \dot \Psi_p(u)\right) \frac{\Delta s_2}{s_1} \\ &\qquad - \frac 1 2 \log\sum_{j,k} d_{j,k} \left(\xi_1, \xi_2\right) \frac{\Delta s_j \Delta s_k}{s_1} \\
		&\leq -2\log\lambda \left( \Psi_p(u) - \alpha s_1 s_2 \dot \Psi_p(u)\right)\frac{\Delta s_2}{s_1} - \frac{\log\lambda}2\sum_{j,k} d_{j,k} \left(\xi_1, \xi_2\right) \frac{\Delta s_j \Delta s_k}{s_1}.
		\end{align*}
		Observe that
		\begin{align*}
		\frac{\Psi_p}{\dot{\Psi}_p} - \alpha s_1 s_2 &= \frac p{p-2}\left(s_1^2 + s_2^2\right) - \alpha s_1 s_2 \geq \left( \frac p{p-2} - \frac \alpha 2 \right) \left(s_1^2 + s_2^2 \right)\\ 
		&\geq \frac{p(2-\alpha)}{2(p-2)} \left(s_1^2 + s_2^2\right).
		\end{align*}
		It follows, in particular, that
		\[
		\Psi_p(u) - \alpha s_1 s_2 \dot{\Psi}_p(u) \geq \left(\frac p 2 \right)^{(2p-4)/p} \frac{2-\alpha}2\left(s_1^2 + s_2^2\right)^{(p-2)/p}.
		\]
		Furthermore, applying the inequality in (\ref{dij-sum-est}), we get: 
		\begin{align*}
		\frac d{dt} \left( \frac{\Delta s_2}{s_1}\right) &\leq -\log\lambda \left(\frac p 2\right)^{(2p-4)/p} (2-\alpha)s_1^{(2p-4)/p} \frac{\Delta s_2}{s_1} \\
		&\quad + \log\lambda \left( \frac p 2 \right)^{(2p-4)/p} s_1^{(2p-4)/p} \frac{12(p-2)}p \left( \frac{\xi_1^2 + \xi_2^2}{s_1^2}\right)^{(p-4)/2p} \left(\frac{\Delta s_2}{s_1}\right)^2.
		\end{align*}
		In particular, if we denote $\chi(t) = \frac{\Delta s_2}{s_1}(t)$, we find that
		\begin{equation}\label{chi-est}
		\frac{d\chi}{dt} \leq -\log\lambda \left(\frac p 2\right)^{(2p-4)/p} s_1^{(2p-4)/p} \chi \left( 2-\alpha-\frac{12(p-2)}{p}\left( \frac{\xi_1^2 + \xi_2^2}{s_1^2}\right)^{(p-4)/2p}\chi\right).
		\end{equation}
		Recall that $\min\left\{s_j, \tilde s_j \right\} \leq \xi_j \leq \max\left\{s_j, \tilde s_j \right\}$, and that $\Delta s_j = \tilde s_j - s_j$ for $j=1,2$. Therefore, 
		\[
		s_j - |\Delta s_j| \leq \xi_j \leq s_j + |\Delta s_j|.
		\]
		In particular, since $|\Delta s_1| \leq \alpha \Delta s_2$ by assumption (1), we get: 
		\[
		\xi_1^2 + \xi_2^2 \geq \xi_1^2 \geq \left(s_1 - |\Delta s_1|\right)^2 \geq \left(s_1 - \alpha \Delta s_2\right)^2 = s_1^2 \left( 1 - \frac{\alpha \Delta s_2}{s_1}\right)^2 \geq s_1^2(1-\chi)^2.
		\]
		Furthermore, since $s_2(t) \leq s_1(t)$ whenever $T_1 \leq t \leq T$, we get: 
		\[
		\frac{\xi_1^2 + \xi_2^2}{s_1^2} \leq  \left( 1 + \frac{|\Delta s_1|}{s_1}\right)^2 + \left(\frac{s_2}{s_1} + \frac{\Delta s_2}{s_1}\right)^2 \leq \left(1+\alpha\chi\right)^2 + (1+\chi)^2 < 2(1+\chi)^2.
		\]
		It follows that: 
		\[
		\left( \frac{\xi_1^2 + \xi_2^2}{s_1^2}\right)^{(p-4)/2p} \leq \begin{cases}
		(1-\chi)^{(p-4)/p}, & p = 3,4; \\
		2^{(p-4)/2p}(1+\chi)^{(p-4)/p}, & p \geq 5.
		\end{cases}
		\]
		Since $s_1(T_1) = s_2(T_1)$, by the first estimate in this lemma and assumption (2), we find that: 
		\[
		0 \leq \chi(T_1) = \frac{\Delta s_2(T_1)}{s_1(T_1)} = \frac{\Delta s_2(T_1)}{s_2(T_1)} \leq \frac{\Delta s_2(0)}{s_2(0)} \leq \frac{1-\alpha}{72}.
		\]
		Again, applying assumption (2) gives us: 
		\[
		2-\alpha-\frac{12(p-2)}{p}\left(\frac{\xi_1^2 + \xi_2^2}{s_1^2}\right)^{(p-4)/2p}\chi(T_1) \geq \frac{1-\alpha}2.
		\]
		So (\ref{chi-est}) now becomes
		\begin{equation}\label{chi-est-T1}
		\frac{d\chi}{dt}\bigg|_{t=T_1} < -\frac{(1-\alpha)\log\lambda}{2}\left(\frac p 2\right)^{(2p-4)/p} s_1(T_1)^{(2p-4)/p} \chi(T_1) < 0.
		\end{equation}
		Repeating the argument for the first estimate in this lemma, we find that the inequalities in (\ref{chi-est-T1}) hold for all $t \in [T_1, T]$. For $T_1 \leq t \leq b \leq T$, by Gr\"onwall's inequality and inequality (d) in Lemma \ref{s1 and s2 approx}, we get: 
		\begin{align*}
		\chi(t) &\leq \chi(T_1) \exp\left( -\frac{(1-\alpha)\log\lambda}{2}\left(\frac p 2\right)^{(2p-4)/p} \int_{T_1}^t s_1(\tau)^{(2p-4)/p}\,d\tau\right) \\
		&\leq \chi(T_1) \exp\left( -\frac{(1-\alpha)\log\lambda}{2}\left(\frac p 2\right)^{(2p-4)/p} s_1(b)^{(2p-4)/p} \right. \\
		&\qquad \qquad \quad \times \left. \int_{T_1}^t \left( 1 + 2^{(p-2)/p} C_0 s_1(T_1)^{(2p-4)/p}s_1(T_1)^{(2p-4)/p}(b-\tau)\right)^{-1}\,d\tau\right) \\
		&= \chi(T_1) \exp\left( \frac{p(1-\alpha)}{2^{(3p-2)/p}(p-2)}\log\left(\frac{1+2^{(p-2)/p}C_0 s_1(T_1)^{(2p-4)/p}(b-t)}{1+2^{(p-2)/p}C_0 s_1(T_1)^{(2p-4)/p}(b-T_1)}\right)\right) \\
		&= \chi(T_1) \left(\frac{1+2^{(p-2)/p}C_0 s_1(T_1)^{(2p-4)/p}(b-t)}{1+2^{(p-2)/p}C_0 s_1(T_1)^{(2p-4)/p}(b-T_1)}\right)^{\beta p/(p-2)}.% \numberthis \label{chi-upper}
		\end{align*}
		The second estimate now follows. %The third estimate now follows after observing that $0 < 1-C_0 s_1(T_1)^{(2p-4)/p}(t-T_1) < 1$ for $T_1 \leq t \leq T$ and applying inequality (e) in lemma \ref{s1 and s2 approx}. 
		
		To prove the final inequality, (\ref{dkappa-approx-2}) and (\ref{chi-est-T1}) show that $\kappa(a) \geq \kappa(T_1)$ and $\chi(T_1) \geq \chi(b)$ for $0 \leq a \leq T_1 \leq b \leq T$. More explicitly, 
		\[
		\frac{\Delta s_2(T_1)}{s_2(T_1)} \leq \frac{\Delta s_2(a)}{s_2(a)} \quad \textrm{and} \quad \frac{\Delta s_2(b)}{s_1(b)} \leq \frac{\Delta s_2(T_1)}{s_2(T_1)}.
		\]
		Recalling that $s_2(T_1) = s_1(T_1)$, combining the above inequalities gives us:
		\[
		\Delta s_2(b) \leq \frac{s_1(b) \Delta s_2(T_1)}{s_2(T_1)} \leq \frac{s_1(b)\Delta s_2(a)}{s_2(a)}.
		\]
		By the assumption that $|\Delta s_1| \leq \alpha\Delta s_2$, we get 
		\[
		\Delta s_2 \leq \norm{\Delta s} \leq \sqrt{1+\alpha^2} \Delta s_2,
		\]
		and combining this inequality with the preceding one gives us the final inequality in the statement of the lemma. 
	\end{proof}
	
	Our final estimate concerns the size of the angles between tangent vectors in the unstable cones near the singularities. This will be used in examining the distance between the unstable subspaces of nearby points in neighborhoods of the singularities. 
	
	Recall the neighborhood $\tilde\Us_1$ of $S$ is given by $\tilde\Us_1 = \union_{k=1}^m \oldphi_k^{-1}\left(D_{\tilde r_1}\right)$. For $x \in \tilde\Us_1$, define:
	\begin{equation}\label{gamma-def}
	\gamma(x) = \max_{\substack{v,w \in K^+(x) \\ \norm{v} = \norm{w} = 1}} \left\{ \frac{\angle\left(Dg_x v, Dg_x w\right)}{\angle(v,w)}\right\} 
	\end{equation}
	and denote $\gamma_j(x) = \gamma(g^j(x))$ for $j \geq 0$. 
	
	\begin{lemma}\label{angle-product}
		For every $x \in \tilde\Us_1$ with $g^j(x)$ in the same component of $\tilde\Us_1$ for $j = 0, \ldots, k$, we have:
		\[
		\prod_{j=0}^{k-1} \gamma_j(x) \leq \left(1+C_0 s_2(0)^{(2p-4)/p}k\right)^{-p/(p-2)},
		\]
		where $C_0$ is the constant from Lemma \ref{s1 and s2 approx}.
	\end{lemma}
	
	\begin{proof}
		Denote $z = \Phi_{kj}(\oldphi_k(x)) = (s_1(0), s_2(0))$, so that $$\left( \Phi_{kj} \circ \oldphi_k \right) \left(g^j(x)\right)  = (s_1(j), s_2(j)).$$ Consider a tangent vector $v = (\zeta_1, \zeta_2)$ in $\C$ along a trajectory of the vector field (\ref{slowdown vector field}). Reparametrizing $\eta = \zeta_2/\zeta_1$ with respect to $s_1$ instead of $t$ along this curve, equation (\ref{tangent-change}) implies 
		\[
		\frac{d\eta}{ds_1} = \frac{d\eta}{dt}\left(\frac{ds_1}{dt}\right)^{-1} = -2\left(\left(1+\eta^2\right)s_2\frac{\dot{\Psi}_p(u)}{\Psi_p(u)} + \left(\frac{1}{s_1} \dot{\Psi}_p(u) + \frac{s_1^2 + s_2^2}{s_1} \frac{\dot{\Psi}_p(u)}{\Psi_p(u)}\right)\eta\right).
		\]
		For $i=1,2$, let $\eta_i(s_1) = \eta_i(s_1, s_1(j), \eta_i^0)$ be a solution to this differential equation with initial condition $\eta_i(s_1(j)) = \eta_i^0$. Then, 
		\[
		\frac{d}{dt}\left(\eta_1 - \eta_2\right) = -2\frac 1{s_1}\left(1+\frac{\dot{\Psi}_p(u)}{\Psi_p(u)} \left(s_1^2 + s_2^2 + s_1s_2(\eta_1 + \eta_2)\right)\right)\left(\eta_1 - \eta_2\right).
		\]
		If $(\xi_1, \xi_2) = D\left(\Phi_{kj}\circ\oldphi_k\right)^{-1}_z(\zeta_1, \zeta_2) \in K^+(x)$, then $|\eta_i| < \alpha < 1$ for $i=1,2$ (see Lemma \ref{strong-cones}), so $\eta_1 + \eta_2 > -2$. Positivity of $\Psi_p$ and $\dot{\Psi}_p$ now yields:
		\[
		\frac{d}{dt}\left(\eta_1 - \eta_2\right) \leq -2\frac 1{s_1} \left(1+\frac{\dot{\Psi}_p(u)}{\Psi_p(u)} \left(s_1 - s_2\right)^2\right)\left(\eta_1 - \eta_2\right),
		\]
		and so by Gr\"onwall's inequality, 
		\begin{align*}
		|\eta_1\left(s_1(j+1)\right) &- \eta_2\left(s_1(j+1)\right)| \\
		&\leq \left|\eta_1^0 - \eta_2^0\right| \exp\left( -2\int_{s_1(j)}^{s_1(j+1)} \frac 1{s_1} \left(1+\frac{\dot{\Psi}_p(u)}{\Psi_p(u)} \left(s_1 - s_2\right)^2\right)ds_1\right) \\
		&\leq \left|\eta_1^0 - \eta_2^0\right| \exp\left( -2\int_{s_1(j)}^{s_1(j+1)} \frac{ds_1}{s_1}\right) \\
		&= \left|\eta_1^0 - \eta_2^0\right| \left(\frac{s_1(j)}{s_1(j+1)}\right)^2 \\
		&= \left|\eta_1^0 - \eta_2^0\right| \left(\frac{s_2(j+1)}{s_2(j)}\right)^2 ,
		\end{align*}
		where the final equality follows from the fact that the trajectories lie on hyperbolas, and so the product $s_1s_2$ is constant. Observe that if $v = (v_1, v_2)$ and $w = (w_1, w_2)$ are two vectors with $\eta_v = v_2/v_1$ and $\eta_w = w_2/w_1$, then 
		\[
		\angle(v,w) = \left|\arctan\eta_v - \arctan\eta_w\right|, 
		\]
		and so by concavity of $\eta \mapsto \arctan\eta$ and conformality of the coordinate map $\Phi_{kj} \circ \oldphi_k$,% : S_{kj}^s \cap \phi_k^{-1}(D_{\tilde r_1}) \to \C$, 
		\begin{align*}
		\gamma_j(x) &\leq \max_{\eta_1, \eta_2} \left\{\frac{ \left| \eta_1(s_1(j+1), s_1(j), \eta_1^0) - \eta_2(s_1(j+1), s_1(j), \eta_2^0)\right|}{\left|\eta_1^0 - \eta_2^0\right|} \right\} \leq \left(\frac{s_2(j+1)}{s_2(j)}\right)^2.
		\end{align*}
		It follows that
		\[
		\prod_{j=0}^{k-1} \gamma_j(x) \leq \left(\frac{s_2(k)}{s_2(0)}\right)^2.
		\]
		The desired result now follows from inequality (b) in Lemma \ref{s1 and s2 approx}, since by hypothesis $g^j(x)$ is in the same component of $\tilde U_1$, hence $G_p^j(z) \in D_{\tilde r_1}$ for $0 \leq j \leq k$. 
	\end{proof}
	
	%In addition to the behavior of trajectories near the singularities, we will also need strong estimates on the invariance of stable and unstable cones. For $x \in M \setminus S$, define the families of cones $K^+(x)$ and $K^-(x)$ by: 
	%\begin{align*}
	%K^+(x) &= \left\{v = (v_1, v_2) : |v_2| < \mu |v_1|\right\}, \\
	%K^-(x) &= \left\{v = (v_1, v_2) : |v_1| < \mu |v_2|\right\}
	%\end{align*}
	%where $\mu$ is the constant of lemma \ref{spread-in-slowdown lemma}, and $(v_1, v_2)$ are the coordinates on $T_x M$ defined in the discussion preceding lemma \ref{regular-cones}. In the original construction of pseudo-Anosov diffeomorphisms yielding lemma \ref{regular-cones}, we have $\mu = 1$. But for certain later arguments, we will require $\mu < 1$. The arguments in constructing this cone are standard, but are stated here for convenience of the reader. 
	%
	%\begin{lemma}\label{strong-cones}
	%	There exists a $0 < \mu_0 < 1$ such that for all $\mu_0 < \mu < 1$, and for all $x \in M$, 
	%	\begin{align*}
	%	Dg_x K^+(x) \subseteq K^+(g(x)) \quad \textrm{and} \quad Dg^{-1}_{g(x)} K^-(g(x)) \subseteq K^-(x).
	%	\end{align*}
	%\end{lemma}
	%
	%\begin{proof}
	%	
	%\end{proof}
	
	\section{Thermodynamics of Young diffeomorphisms}
	
	Given a $C^{1+\alpha}$ diffeomorphism $f$ on a compact Riemannian manifold $M$, we call an embedded $C^1$ disc $\gamma \subset M$ an \emph{unstable disc} (resp. \emph{stable disc}) if for all $x, y \in \gamma$, we have $d(f^{-n}(x), f^{-n}(y)) \to 0$  (resp. $d(f^n(x), f^n(y)) \to 0$) as $n \to +\infty$. A collection of embedded $C^1$ discs $\Gamma = \{\gamma_i\}_{i \in \mathcal I}$ is a \emph{continuous family of unstable discs} if there is a Borel subset $K^s \subset M$ and a homeomorphism $\Phi : K^s \times D^u \to \union_i \gamma_i$, where $D^u \subset \R^d$ is the closed unit disc for some $d < \dim M$, satisfying: 
	\begin{itemize}
		\item The assignment $x \mapsto \Phi|_{\{x\} \times D^u}$ is a continuous map from $K^s$ to the space of $C^1$ embeddings $D^u \hookrightarrow M$, and this assignment can be extended to the closure $\overline{K^s}$; 
		\item For every $x \in K^s$, $\gamma = \Phi(\{x\} \times D^u)$ is an unstable disc in $\Gamma$.
	\end{itemize}
	Thus the index set $\mathcal I$ may be taken to be $K^s \times \{0\} \subset K^s \times D^u$. We define \emph{continuous families of stable discs} analogously. 
	
	A subset $\Lambda \subset M$ has \emph{hyperbolic product structure} if there is a continuous family $\Gamma^u = \{\gamma^u_i\}_{i \in \mathcal I}$ of unstable discs and a continuous family $\Gamma^s = \{\gamma^s_j\}_{j \in \mathcal J}$ of stable discs such that
	\begin{itemize}
		\item $\dim \gamma^u_i + \dim\gamma^s_j = \dim M$ for all $i,j$; 
		\item the unstable discs are transversal to the stable discs, with an angle uniformly bounded away from 0; 
		\item each unstable disc intersects each stable disc in exactly one point; 
		\item $\Lambda = \big( \union_i \gamma^u_i\big) \cap \big(\union_j \gamma^s_j \big)$. 
	\end{itemize}
	
	A subset $\Lambda_0 \subset \Lambda$ with hyperbolic product structure is an \emph{s-subset} if the continuous family of unstable discs defining $\Lambda_0$ is the same as the continuous family of unstable discs for $\Lambda$, and the continuous family of stable discs defining $\Lambda_0$ is a subfamily $\Gamma_0^s$ of the continuous family of stable discs defining $\Gamma_0$. In other words, if $\Lambda_0 \subset \Lambda$ has hyperbolic product structure generated by the families of stable and unstable discs given by $\Gamma_0^s$ and $\Gamma_0^u$, then $\Lambda_0$ is an $s$-subset if $\Gamma_0^s \subseteq \Gamma^s$ and $\Gamma_0^u = \Gamma^u$. A \emph{u-subset} is defined analogously. 
	
	\begin{definition}\label{Young tower def}
		A $C^{1+\alpha}$ diffeomorphism $f : M \to M$, with $M$ a compact Riemannian manifold, is a \emph{Young's diffeomorphism} if the following conditions are satisfied: 
		\begin{enumerate}[label=(Y\arabic*)]
			\item There exists $\Lambda \subset M$ (called the \emph{base}) with hyperbolic product structure, a countable collection of continuous subfamilies $\Gamma_i^s \subset \Gamma^s$ of stable discs, and positive integers $\tau_i$, $i \in \N$, such that the $s$-subsets
			\[
			\Lambda_i^s := \union_{\gamma \in \Gamma^s_i} \big(\gamma \cap \Lambda \big) \subset \Lambda
			\]
			are pairwise disjoint and satisfy:
			\begin{enumerate}[label=(\alph*)]
				\item \emph{invariance}: for $x \in \Lambda_i^s$, 
				\[
				f^{\tau_i}(\gamma^s(x)) \subset \gamma^s(f^{\tau_i}(x)), \quad \textrm{and} \quad f^{\tau_i}(\gamma^u(x)) \supset \gamma^u(f^{\tau_i}(x)),
				\]
				where $\gamma^{u,s}(x)$ denotes the (un)stable disc containing $x$; and, 
				\item \emph{Markov property}: $\Lambda_i^u := f^{\tau_i}(\Lambda_i^s)$ is a $u$-subset of $\Lambda$ such that for $x \in \Lambda_i^s$, 
				\[
				f^{-\tau_i}(\gamma^s(f^{\tau_i}(x)) \cap \Lambda_i^u) = \gamma^s(x) \cap \Lambda, \quad \textrm{and} \quad f^{\tau_i} (\gamma^u(x) \cap \Lambda_i^s) = \gamma^u(f^{\tau_i}(x)) \cap \Lambda. 
				\]
			\end{enumerate}
			\item For $\gamma^u \in \Gamma^u$, we have
			\[
			\mu_{\gamma^u}(\gamma^u \cap \Lambda) > 0, \quad \textrm{and} \quad \mu_{\gamma^u}\Big( \cl\big( \left(\Lambda \setminus \textstyle\union_i \Lambda_i^s\right) \cap \gamma^u\big)\Big) = 0,
			\]
			where $\mu_{\gamma^u}$ is the induced Riemannian leaf volume on $\gamma^u$ and $\cl(A)$ denotes the closure of $A$ in $M$ for $A \subseteq M$. 
			\item There is $a \in (0,1)$ so that for any $i \in \N$, we have:
			\begin{enumerate}[label=(\alph*)]
				\item For $x \in \Lambda_i^s$ and $y \in \gamma^s(x)$, 
				\[
				d(F(x), F(y)) \leq ad(x,y);
				\]
				\item For $x \in \Lambda_i^s$ and $y \in \gamma^u(x) \cap \Lambda_i^s$, 
				\[
				d(x,y) \leq ad(F(x), F(y)),
				\]
			\end{enumerate}
			where $F : \union_i \Lambda_{i}^s \to \Lambda$ is the \emph{induced map} defined by 
			\[
			F|_{\Lambda^s_i} := f^{\tau_i}|_{\Lambda^s_i}.
			\]
			\item Denote $J^u F(x) = \det\big|DF|_{E^u(x)}\big|$. There exist $c > 0$ and $\kappa \in (0,1)$ such that: 
			\begin{enumerate}[label=(\alph*)]
				\item For all $n \geq 0$, $x \in F^{-n}\left(\union_i \Lambda_i^s\right)$ and $y \in \gamma^s(x)$, we have 
				\[
				\left| \log \frac{J^u F(F^n(x))}{J^u F(F^n(y))}\right| \leq c\kappa^n;
				\]
				\item For any $i_0, \ldots, i_n \in \N$ with $F^k(x), F^k(y) \in \Lambda^s_{i_k}$ for $0 \leq k \leq n$ and $y \in \gamma^u(x)$, we have 
				\[
				\left| \log\frac{J^u F(F^{n-k}(x))}{J^u F(F^{n-k}(y))}\right| \leq c\kappa^k.
				\]
			\end{enumerate}
			\item There is some $\gamma^u \in \Gamma^u$ such that 
			\[
			\sum_{i=1}^\infty \tau_i \mu_{\gamma^u} \left(\Lambda_i^s\right) < \infty. 
			\]
		\end{enumerate}
	\end{definition}
	
	%We also define the \emph{Young tower} $(X, \tilde F)$, where $X = \coprod_{i \geq 0} \tau^{-1}([i,\infty))$, $\tau : \Lambda \to \Z_{\geq 0}$ is the inducing time $\tau(x) = \tau_i$ for $x \in \Lambda_i^s$, and $\tilde F : X \to X$ is the map defined by 
	%\[
	%\tilde F(x,i) = \begin{cases}
	%(x,i+1) & \textrm{if } i < \tau(x), \\
	%(f^{\tau(x)}(x),0) & \textrm{if } i = \tau(x).
	%\end{cases}
	%\]
	%Here we identify $X$ as a subset of $\Lambda \times \Z_{\geq 0}$ by considering $(x,i)$ as an element of $\tau^{-1}([i,\infty)) \subset X$. 
	
	We say the tower satisfies the \emph{arithmetic condition} if the greatest common divisor of the integers $\{\tau_i\}$ is 1. 
	
	We use the following result to discuss thermodynamics of Young's diffeomorphisms, which was originally presented as Proposition 4.1 and Remark 4 in \cite{PSZ17}. 
	
	\begin{proposition}\label{Young tower nuke}
		Let $f : M \to M$ be a $C^{1+\alpha}$ diffeomorphism of a compact smooth Riemannian manifold $M$ satisfying conditions (Y1)-(Y5), and assume $\tau$ is the first return time to the base of the tower. Then the following hold: 
		\begin{enumerate}[label=(\arabic*)]
			\item There exists an equilibrium measure $\mu_1$ for the potential $\phi_1$, which is the unique SRB measure. 
			\item Assume that for some constants $C>0$ and $0 < h < h_{\mu_1}(f)$, with $h_{\mu_1}(f)$ the metric entropy, we have
			\[
			S_n := \# \left\{ \Lambda_i^s : \tau_i = n \right\} \leq Ce^{hn}
			\]
			Define 
			\begin{equation}\label{lambda_1 def}
			\log\lambda_1 = \sup_{i \geq 1} \sup_{x \in \Lambda_i^s} \frac 1{\tau_i} \log\left| J^u F(x)\right| \leq \max_{x \in M} \log\left|J^u f(x)\right|,
			\end{equation}
			and
			\begin{equation}\label{t_0 def}
			t_0 = \frac{h-h_{\mu_1}(f)}{\log\lambda_1 - h_{\mu_1}(f)}.
			\end{equation}
			Then for every $t \in (t_0, 1)$, there exists a measure $\mu_t \in \mathcal M(f,Y)$, where $Y = \left\{f^k(x) : x \in \union \Lambda_i^s, \: 0 \leq k \leq \tau(x)-1 \right\}$, which is a unique equilibrium measure for the potential $\phi_t$. 
			\item Assume that the tower satisfies the arithmetic condition, and that there is $K > 0$ such that for every $i \geq 0$, every $x,y \in \Lambda_i^s$, and any $j \in \{0, \ldots, \tau_i\}$, 
			\begin{equation}\label{4.2 bound}
			d\left(f^j(x), f^j(y)\right) \leq K\max\{d(x,y), d(F(x), F(y))\}.
			\end{equation}
			Then for every $t_0 < t < 1$, the measure $\mu_t$ has exponential decay of correlations and satisfies the Central Limit Theorem with respect to a class of functions which contains all H\"older continuous functions on $M$. 
		\end{enumerate}
	\end{proposition}
	
	\section{Young towers over pseudo-Anosov diffeomorphisms}
	
	Our argument that smooth pseudo-Anosov diffeomorphisms are Young's diffeomorphisms requires the construction of a hyperbolic tower on pseudo-Anosov homeomoprhisms first. We begin this section by constructing this hyperbolic tower, taking an element of the Markov partition of the pseudo-Anosov homeomorphism as the base of the tower. 
	
	We assume that our pseudo-Anosov homeomorphism $f$ admits only one singularity; the analysis follows similarly with more singularities, but the notation becomes unwieldy due to the different numbers of prongs at each singularity. Therefore we state without proof that the arguments of this section imply that pseudo-Anosov diffeomorphisms admitting multiple singularities are also Young diffeomorphisms. An example of a pseudo-Anosov homeomorphism of the genus-2 torus admitting only one singularity may be found in \cite{Penn88}.
	
	By Proposition \ref{pseudo-Anosov Markov}, a pseudo-Anosov surface homeomorphism $f : M \to M$ admits a Markov partition of arbitrarily small diameter. Let $\tilde\Ps$ be such a Markov partition, and let $\tilde P \in \tilde \Ps$ be an element of the Markov partition contained in a chart $U_1$ not intersecting with the chart $U_0$ of the singularity $x_0$. For $x \in \tilde P$, let $\tilde\gamma^s(x)$ and $\tilde\gamma^u(x)$ respectively be the connected component of the intersection of the stable and unstable leaves with $\tilde P$ containing $x$.
	
	Let $\tilde\tau(x)$ be the first return time of $x$ to $\mathrm{Int}\tilde P$ for $x \in \tilde P$. For $x$ with $\tilde \tau(x) < \infty$, define:
	\[
	\tilde\Lambda^s(x) = \union_{y \in \tilde U^u(x) \setminus \tilde A^u(x)} \tilde\gamma^s(y),
	\]
	where $\tilde U^u(x) \subseteq \tilde\gamma^u(x)$ is an interval containing $x$, open in the induced topology of $\tilde\gamma(x)$, and $\tilde A^u(x) \subset \tilde U^u(x)$ is the set of points that either lie on the boundary of the Markov partition, or never return to $\tilde P$. One can show the leaf volume of $\tilde A^u(x)$ is 0, so that for each $y \in \tilde\Lambda^s(x)$, the leaf volume of $\tilde\gamma(y) \cap \tilde\Lambda^s(x)$ is positive. We further choose our interval $U^u(x)$ so that
	\begin{itemize}
		\item for $y \in \tilde \Lambda^s(x)$, we have $\tilde\tau(y) = \tilde\tau(x)$; and, 
		\item for $y \in \tilde P$ with $\tilde\tau(x) = \tilde\tau(y)$, we have $y \in \tilde\Lambda(z)$ for some $z \in \tilde P$. 
	\end{itemize}
	One can show the image under $\tilde f^{\tilde\tau(x)}$ of $\tilde\Lambda^s(x)$ is a $u$-subset containing $\tilde f^{\tilde \tau(x)}(x)$, and that for $x, y \in \tilde P$ with finite return time, either $\tilde\Lambda^s(x)$ and $\tilde\Lambda^s(y)$ are disjoint or coinciding. As discussed in \cite{PSZ17}, this gives us a countable collection of disjoint sets $\tilde\Lambda^s_i$ and numbers $\tilde\tau_i$ for which the pseudo-Anosov homeomorphism $f : M \to M$ is a Young map, with $s$-sets $\tilde\Lambda_i^s$, inducing times $\tilde\tau_i$, and tower base
	\[
	\tilde\Lambda := \union_{i =1}^\infty \cl\big(\tilde\Lambda_i^s\big).
	\]
	
	In the following theorem, Conditions (Y1$'$) through (Y5$'$) are virtually identical to Conditions (Y1) through (Y5) in Definition \ref{Young tower def}. They are reprinted in the following theorem because pseudo-Anosov homeomorphisms are not true diffeomorphisms, and thus by definition cannot satisfy Conditions (Y1) through (Y5). However, analogous conditions may be established for pseudo-Anosov homeomorphisms, and these conditions will be used to show that globally smooth realizations of pseudo-Anosov diffeomorphisms (which are true diffeomorphisms) are Young's diffeomorphisms. 
	
	\begin{theorem}\label{PAH-hyperbolic-tower}
		The set $\tilde \Lambda$ defined above for the pseudo-Anosov homeomorphism $f : M \to M$ satisfies the following conditions: 
		\begin{enumerate}[label=\emph{(Y\arabic*$'$)}]
			\item $\tilde \Lambda$ has hyperbolic product structure, and the sets $\left\{ \tilde \Lambda_i^s\right\}_{i \in \N}$ are pairwise disjoint $s$-subsets and satisfy:
			\begin{enumerate}[label=(\alph*)]
				\item \textbf{invariance}: for $x \in \tilde\Lambda_i^s$, 
				\[
				f^{\tau_i}(\gamma^s(x)) \subset \gamma^s(f^{\tau_i}(x)), \quad \textrm{and} \quad f^{\tau_i}(\gamma^u(x)) \supset \gamma^u(f^{\tau_i}(x)),
				\]
				where $\gamma^{u,s}(x)$ denotes the (un)stable disc containing $x$; and, 
				\item \textbf{Markov property}: $\tilde \Lambda_i^u := f^{\tau_i}(\Lambda_i^s)$ is a $u$-subset of $\tilde \Lambda$ such that for $x \in \tilde \Lambda_i^s$, 
				\[
				f^{-\tau_i}(\gamma^s(f^{\tau_i}(x)) \cap \tilde\Lambda_i^u) = \gamma^s(x) \cap \tilde\Lambda, \quad \textrm{and} \quad f^{\tau_i} (\gamma^u(x) \cap \tilde\Lambda_i^s) = \gamma^u(f^{\tau_i}(x)) \cap \tilde\Lambda. 
				\]
			\end{enumerate}
			\item For $\gamma^u \in \Gamma^u$, we have
			\[
			\nu^s\left(\gamma^u \cap \tilde \Lambda\right) > 0, \quad \textrm{and} \quad \nu^s\Big( \cl\big( \left(\tilde \Lambda \setminus \textstyle\union_i \tilde \Lambda_i^s\right) \cap \gamma^u\big)\Big) = 0,
			\]
			where $\nu^s$ is the transversal invariant measure with respect to the stable foliation $\Fs^s$ for $f$. 
			\item There is $a \in (0,1)$ so that for any $i \in \N$, we have:
			\begin{enumerate}[label=(\alph*)]
				\item For $x \in \tilde\Lambda_i^s$ and $y \in \gamma^s(x)$, 
				\[
				d^s(F(x), F(y)) \leq ad^s(x,y);
				\]
				\item For $x \in \tilde\Lambda_i^s$ and $y \in \gamma^u(x) \cap \tilde \Lambda_i^s$, 
				\[
				d^u(x,y) \leq ad^u(F(x), F(y)),
				\]
			\end{enumerate}
			where $F : \union_i \tilde\Lambda_{i}^s \to \tilde\Lambda$ is the \emph{induced map} defined by 
			\[
			F|_{\tilde\Lambda^s_i} := f^{\tau_i}|_{\tilde\Lambda^s_i}
			\]
			and $d^s$ and $d^u$ are the distances in the stable and unstable leaves of the foliations $\Fs^s$ and $\Fs^u$ in $\tilde P$, given respectively by $\nu^u$ and $\nu^s$.
			\item Denote $J^u F(x) = \det\big|DF|_{E^u(x)}\big|$. There exist $c > 0$ and $\kappa \in (0,1)$ such that: 
			\begin{enumerate}[label=(\alph*)]
				\item For all $n \geq 0$, $x \in F^{-n}\left(\union_i \tilde \Lambda_i^s\right)$ and $y \in \gamma^s(x)$, we have 
				\[
				\left| \log \frac{J^u F(F^n(x))}{J^u F(F^n(y))}\right| \leq c\kappa^n;
				\]
				\item For any $i_0, \ldots, i_n \in \N$ with $F^k(x), F^k(y) \in \tilde \Lambda^s_{i_k}$ for $0 \leq k \leq n$ and $y \in \gamma^u(x)$, we have 
				\[
				\left| \log\frac{J^u F(F^{n-k}(x))}{J^u F(F^{n-k}(y))}\right| \leq c\kappa^k.
				\]
			\end{enumerate}
			\item There is some $\gamma^u \in \tilde \Gamma^u$ such that 
			\[
			\sum_{i=1}^\infty \tau_i \nu^s \big(\tilde \Lambda_i^s \cap \gamma^u\big) < \infty. 
			\]
		\end{enumerate}
	\end{theorem}
	
	\begin{proof}
		Properties (Y1$'$), (Y3$'$), and (Y4$'$) all follow from Proposition \ref{PAH-differential}. Property (Y2$'$) follows because $x \in  \cl\big( \left(\Lambda \setminus \textstyle\union_i \Lambda_i^s\right) \cap \gamma^u\big)$ implies either that $x \in \partial P$ or $\tau(x) = \infty$, both of which happen on a set of Lebesgue measure 0 (and the smooth measure for pseudo-Anosov homeomorphisms has density with respect to Lebesgue measure). And since $\tau$ is a first return time, (Y5$'$) follows from Kac's theorem. 
	\end{proof}
	
	The next lemma gives a bound on the number $S_n$ of distinct $s$-subsets $\tilde \Lambda^s_i$ with a given inducing time$ \tilde\tau_i = n$. Since the pseudo-Anosov homeomorphism $f$ is topologically conjugate to the smooth realization $g$, this will eventually give us an analogous bound on the number of distinct $s$-subsets for the base of the tower for $g$. (See Condition (2) of Proposition \ref{Young tower nuke}.)
	
	\begin{lemma}\label{number-of-inducing-sets}
		There exists $h < \htop(f)$ such that $S_n \leq e^{hn}$, where $S_n$ is the number of $s$-sets $\tilde\Lambda^s_i$ with inducing time $\tilde\tau_i = n$. 
	\end{lemma}
	\begin{proof} The proof is analogous to \cite{PSZ17}, Lemma 6.1, since pseudo-Anosov homeomorphisms admit finite Markov partitions. 
	\end{proof}
	
	Let $H :M \to M$ be the conjugacy map so that $g \circ H = H \circ f$, and let $\Ps = H(\tilde \Ps)$, $P = H(\tilde P)$. Then $\Ps$ is a Markov partition for the pseudo-Anosov diffeomorphism $(M, g)$, and $P$ is a partition element. By continuity of $H$, we may assume the elements of $\Ps$ have arbitrarily small diameter. Further let $\Lambda = H(\tilde \Lambda)$. Then $\Lambda$ has direct hyperbolic product structure with full length stable and unstable curves $\gamma^s(x) = H(\tilde\gamma^s(x))$ and $\gamma^u(x) = H(\tilde\gamma^u(x))$. Then $\Lambda^s_i = H(\tilde\Lambda_i^s)$ are $s$-sets and $\Lambda^u_i = H(\tilde\Lambda^u_i) = g^{\tau_i}(\Lambda_i^s)$, where $\tau_i = \tilde\tau_i$ for each $i$, and $\tau(x) = \tau_i$ whenever $x \in \Lambda_i^s$. 
	
	Recall $\Us_0 = \union_{k=1}^m \oldphi_k^{-1}\left(D_{r_0}\right)$. If there is only one singularity, $\Us_0 = \oldphi_0^{-1} \left(D_{r_0}\right)$. Given $Q>0$, we can take $r_0$ in the construction of $g$ to be so small and refine the partition $\tilde\Ps$ so that the partition element $\tilde P$ (and hence $P$) may be chosen so that
	\begin{equation}\label{Q-definition}
	g^n(x) \not\in\Us_0 \textrm{ for any } 0 \leq n \leq Q
	\end{equation}
	and any $x$ so that either $x \in P$, or $x \not\in \Us_0$ while $g^{-1}(x) \in \Us_0$. 
	
	We now prove the set $\Lambda = H(\tilde\Lambda)$ constructed above is the base of a Young tower on $M$ for the diffeomorphism $g$. Properties (Y1), (Y2), and (Y5) are straightforward to verify. Our strategy in proving these conditions, along with (Y3), is similar to that used in \cite{PSZ17}, but we restate it here for the reader's convenience. The main difference between the argument used for these pseudo-Anosov diffeomorphisms and the Katok map comes in proving (Y4), where we use a local trivialization of our surface $M$ as opposed to the universal cover of $\T^2$ by $\R^2$. 
	
	\begin{theorem}\label{PAD tower}
		The collection of $s$-subsets $\Lambda_i^s = H(\tilde\Lambda_i^s)$ satisfies conditions (Y1) - (Y5), making the smooth pseudo-Anosov diffeomorphism $g : M \to M$ a Young's diffeomorphism. 
	\end{theorem}
	
	\begin{proof}
		Condition (Y1) follows from the corresponding properties of the pseudo-Anosov homeomorphism $f$ since $H$ is a topological conjugacy. The fact that $\mu_{\gamma^u}\left(\gamma^u \cap \Lambda\right) > 0$ follows from the corresponding property for the $\tilde\gamma^u$ leaves. Suppose $x \in \cl\big(\left(\Lambda \setminus \union_i \Lambda_i^s\right)\cap \gamma^u\big)$. Then either $x$ lies on the boundary of the Markov partition element $P$, or $\tau(x) = \infty$, and since both the Markov partition boundary and the set of $x \in P$ with $\tau(x) = \infty$ are Lebesgue null, we get condition (Y2). Condition (Y5) follows from Kac's formula, since the inducing times are first return times to the base of the tower. 	
		
		To prove condition (Y3), define the \emph{itinerary} $\mathcal I(x) = \{0 = n_0 < n_1 < \cdots < n_{2L+1} = \tau(x)\} \subset \Z$ of a point $x \in \Lambda$, with $L = L(x)$, so that $g^k(x) \in \Us_0$ if and only if $n_{2j-1} \leq k < n_{2j}$ for $j \geq 1$. Assume $\Lambda$ is small enough so that $\mathcal I(x) = \mathcal I(y)$ whenever $y \in \gamma(x) \subset \Lambda$.
		
		Let $x \in \Lambda_i^s$, $y \in \gamma^s(x) \subset \Lambda_i^s$. Denote $x_n = g^n(x)$ and $y_n = g^n(y)$. Note $\gamma^s(x) \subset \Fs^s(x)$. By invariance of the stable and unstable measured foliations $\Fs^s$ and $\Fs^u$, $y_n$ lies on the stable curve $\Fs^s(x_n)$ through $x_n$ for every $n \geq 1$. For $n_{2j} \leq n < n_{2j+1}$, $T_{x_n}\Fs^s(x_n) = E^s_{x_n}$ lies inside $C^-_x$; in fact one can show that $\Fs^s(x_n)$ is an admissible manifold. Thus the segment of $\Fs^s(x_n)$ joining $x_n$ and $y_n$ expands uniformly under the homeomorphism $f^{-1}$. Due to our choice of the number $Q$, there is a number $\beta \in(0,1)$ such that 
		\begin{equation}\label{outer-lipschitz}
		d\left(x_{n_{2j+1}}, y_{n_{2j+1}}\right) \leq \beta^{n_{2j+1} - n_{2j}} d\left(x_{n_{2j}}, y_{n_{2j}}\right) \leq \beta^Q d\left(x_{n_{2j}}, y_{n_{2j}}\right). 
		\end{equation}
		Now we consider $n_{2j-1} \leq n < n_{2j}$. Let $\left[ m_j^1, m_j^2\right] \subseteq \left[ n_{2j-1}, n_{2j}-1\right]$ be the largest interval (possibly empty) with $x_n$ in the closure of $\tilde\Us_1 = \oldphi_0^{-1}\left(D_{\tilde r_1}(0)\right)$ for every $n \in \left[ m_j^1, m_j^2\right]$. By virtue of Lemma \ref{annular bound in M}, there is a uniform $T > 0$ with $m^1_j - n_{2j-1} \leq T$ and $n_{2j} -m^2_j \leq T$. Thus there is a constant $C>0$ so that 
		\begin{equation}\label{annular-lipschitz}
		d\big(x_{m_j^1}, y_{m_j^1}\big) \leq Cd\big(x_{n_{2j-1}}, y_{n_{2j-1}}\big) \quad \textrm{and} \quad d\big(x_{n_{2j}}, y_{n_{2j}} \big) \leq Cd\big(x_{m^2_j}, y_{m^2_j}\big).
		\end{equation}
		Now, let $s(t)$ and $\tilde s(t)$ be solutions to equation (\ref{slowdown vector field}) with $s(0) = x_{m_j^1}$ and $\tilde s(0) = y_{m_j^1}$. Assumption (1) of Lemma \ref{spread-in-slowdown lemma} is satisfied since $y_n$ lies in the stable cone of $x_n$ for every $n$, and Assumption (2) can be assured if our choice of $r_0$ in the slowdown construction of the pseudo-Anosov diffeomorphism is chosen to be sufficiently small. So by the final inequality of this lemma, letting $a = m_j^1$ and $b = m_j^2$, we get: 
		\[
		\norm{\Delta s\left(m_j^2\right)} \leq \sqrt{1+\alpha^2} \frac{s_1\left(m_j^2\right)}{s_2\left(m_j^1\right)} \norm{\Delta s\left(m_j^1\right)}. 
		\]
		Let $\Delta_{kj} s(t) = \Phi^{-1}_{kj}\left(\tilde s(t)\right) - \Phi^{-1}_{kj}\left(s(t)\right)$. Because $\Phi_{kj}$ is uniformly bounded above and below, there is a constant $K > 0$ such that for every $t$ for which $\tilde s(t)$ and $s(t)$ are defined, 
		\begin{equation}\label{norm-equivalence}
		K^{-1} \norm{\Delta_{kj}s(t)} \leq \norm{\Delta s(t)} \leq K\norm{\Delta_{kj}s(t)},
		\end{equation}
		and since the Riemannian metric in $\Us_0$ is given in coordinates by $dt_1^2 + dt_2^2 = \left(\Phi_{kj}^{-1}\right)^*\left(ds_1^2 + ds_2^2\right)$, we get $\norm{\Delta_{kj} s(n)} = d\left(x_n, y_n\right)$ for $n \in \left[m_j^1, m_j^2\right]$. Therefore, combining this observation with (\ref{norm-equivalence}), (\ref{outer-lipschitz}), (\ref{annular-lipschitz}), and (\ref{deviation bound}), we get: 
		\begin{align*}
		d\left(x_{n_{2j}}, y_{n_{2j}}\right) &\leq CK^2\sqrt{1+\alpha^2} \frac{s_1\left(m_j^2\right)}{s_2\left(m_j^1\right)}d\left(x_{m_j^1}, y_{m_j^1}\right) \\
		&\leq C^2K^2\sqrt{1+\alpha^2}\frac{s_1\left(m_j^2\right)}{s_2\left(m_j^1\right)}d\left(x_{n_{2j-1}}, y_{n_{2j-1}}\right) \\
		&\leq C^2K^2\beta^Q \sqrt{1+\alpha^2}\frac{s_1\left(m_j^2\right)}{s_2\left(m_j^1\right)}d\left(x_{n_{2j-2}}, y_{n_{2j-2}}\right) .
		\end{align*}
		Since $s_1\left(m_j^2\right)$ and $s_2\left(m_j^1\right)$ are each of order $r_0$, their quotient is uniformly bounded, so assuming $Q$ is sufficiently large, there is a $0 < \theta_1 < 1$ for which 
		\begin{equation}\label{inner-slowdown-lipschitz}
		d\left(x_{n_{2j}}, y_{n_{2j}}\right) \leq \theta_1 d\left(x_{n_{2j-2}}, y_{n_{2j-2}}\right) 
		\end{equation}
		and a similar bound holds for odd indices of the itinerary. It follows that
		\[
		d\left(g^{\tau(x)}(x), g^{\tau(x)}(y)\right) \leq \theta_1^L d(x,y),
		\]
		where $L$ is determined by the itinerary $\mathcal I(x)$. Condition (Y3a) follows, and (Y3b) follows by the same argument applied to $g^{-1}$. 
		
		To prove condition (Y4), we prove condition (Y4a) and note that (Y4b) can be proved similarly by considering $g^{-1}$ instead of $g$. We use the following general statement, originally presented as Lemma 6.3 in \cite{PSZ17}:
		
		\begin{lemma}\label{general-statement}
			Let $\{A_n\}$, $\{B_n\}$, $0 \leq n \leq N$, be two collections of linear transformations of $\R^d$. Given a subspace $E \subset \R^d$, let $K = K(E,\theta)$ denote the cone of angle $\theta$ around $E$. Assume the subspace $E$ is such that:
			\begin{enumerate}[label=(\alph*)]
				\item $A_n(K) \subset K$ for all $n$; 
				\item There are $\gamma_n > 0$ such that for each $n$, and for any unit vectors $v, w \in K$,
				\[
				\angle\left(A_n v, A_nw\right) \leq \gamma_n \angle(v,w);
				\]
				\item There are $d > 0$ and $\delta_n > 0$ such that for each $n \geq 0$, and every $v \in K$,
				\[
				\norm{A_n v- B_n v} \leq d\delta_n \norm{A_n v}; 
				\]
				\item There is $c > 0$ independent of $n$ such that for every $v \in K$, 
				\[
				\norm{A_n v} \geq c\norm v.
				\]
			\end{enumerate}
			Then there is a $C > 0$, independent of the choice of linear transformations $\{A_n\}$ and $\{B_n\}$, such that for every $v, w \in K$, 
			\begin{equation}\label{cocycle-decay}
			\left| \log \frac{\norm{\prod_{n=0}^N A_n v}}{\norm{\prod_{n=0}^N B_n w}}\right| \leq C \left( d\sum_{n=0}^N \delta_n + \angle(v,w) \sum_{n=0}^N \prod_{k=0}^n \gamma_k\right).
			\end{equation}
		\end{lemma}
		Let $x \in P$ with $N := \tau(x) - 1 < \infty$, and let $y \in \gamma^s(x) \subset P$. For each $n \geq 0$, once again let $x_n = g^n(x)$ and $y_n = g^n(y)$, and in each tangent space $T_{x_n}M$, let $K^+_n = K^+(x_n) \subset T_{x_n}M$ denote the cone of angle $\arctan\alpha$ around $E^u(x_n)$ described in Lemma \ref{strong-cones}. By this lemma, the sequence of cones $\left\{K_n^+\right\}$ is invariant under $Dg$. For each $n$, denote $\tilde A_n = Dg_{x_n} : T_{x_n}M \to T_{x_{n+1}}M$ and $\hat B_n = Dg_{y_n} : T_{y_n} M \to T_{y_{n+1}}M$. Further, since $y_n$ lies on the stable leaf of $x_n$ for all $n$, let $P_n : T_{y_n}M \to T_{x_n}M$ denote parallel translation along the segment of the stable leaf connecting $y_n$ to $x_n$, and denote $\tilde B_n = P_{n+1} \circ \hat B_n \circ P_n^{-1} : T_{x_n} M \to T_{x_{n+1}}M$. Using the orthonormal coordinates $(\xi_1, \xi_2)$ for $T_{x_n}M$ defined previously, so that $\xi_1$ denotes the unstable direction and $\xi_2$ denotes the stable direction (see the discussion preceding Proposition \ref{regular-cones}), we may isometrically identify each tangent space $T_{x_n}M$ with $\R^2$ with the Euclidean metric. Call this isometry $\Xi_n : T_{x_n}M \to \R^2$, and denote $A_n = \Xi_{n+1} \circ \tilde A_n \circ \Xi_n^{-1} : \R^2 \to \R^2$ and $B_n = \Xi_{n+1} \circ \tilde B_n \circ \Xi_n^{-1} : \R^2 \to \R^2$. Also let $K = \Xi_n(K^+_n) \subset \R^2$. Since $\Xi_n$ is an isometry and $K^+_n$ is a cone of angle $\arctan \alpha$ for each $n$, $K$ is independent of $n$ and is thus well-defined. Finally, define the numbers $d = d(x,y)$, as well as
		\[
		\gamma_n = \max_{\substack{v,w \in K \\ \norm v = \norm w = 1}} \left\{ \frac{\angle\left( A_n v,  A_n w\right)}{\angle(v,w)}\right\} \quad \textrm{and} \quad \delta_n = \frac 1 d \max_{v \in K \setminus\{0\}} \left\{\frac{\norm{A_n v - B_n v}}{\norm{A_n v}}\right\}
		\]
		for each $n \geq 0$. 
		
		%The final step in proving that our pseudo-Anosov diffeomorphism g is a Young’s	diffeomorphism relies on the following technical lemma. Its proof is somewhat similar to the proof of Lemma 6.4 in [11], but requires some modifications related to subtle differences in the slowdown function used in the Katok map as opposed to our pseudo-Anosov diffeomorphism g as well as to the fact that the universal cover of a surface which is not a torus is not $\mathbb{R}^2$
		
		The final step in proving our pseudo-Anosov diffeomorphism $g$ is a Young's diffeomorphism relies on the following technical lemma. Its proof is somewhat similar to the proof of Lemma 6.4 in \cite{PSZ17}, but requires some modifications related to the subtle differences in the slowdown function used in the Katok map as opposed to our pseudo-Anosov diffeomorphism $g$, as well as to the fact that the universal cover of a surface that is not a torus is not $\R^2$. 
		\begin{lemma}\label{final-step}
			The linear operators $A_n$ and $B_n$, as well as the cone $K$, all satisfy the conditions of Lemma \ref{general-statement} using $\gamma_n$, $\delta_n$, $d$, and $N = \tau(x) - 1$ defined above. Furthermore, there are constants $\tilde C > 0$ and $0 < \theta_2 < 1$, independent of $x \in P$, such that:
			\[
			\sum_{n=0}^{\tau(x)-1} \delta_n < \tilde C, \quad \sum_{n=0}^{\tau(x)-1} \prod_{k=0}^n \gamma_k < \tilde C, \quad \textrm{and} \quad \prod_{n=0}^{\tau(x)-1} \gamma_n < \theta_2. 
			\]
		\end{lemma}
		\begin{proof}[Proof of Lemma \ref{final-step}]
			Condition (a) of Lemma \ref{general-statement} follows from the definition of $A_n$, the invariance of the cone family $K^+_n$ under $\tilde A_n$, and the fact that $\Xi_n : T_{x_n}M \to \R^2$ is an isometry for every $n$. Conditions (b) and (c) of Lemma \ref{general-statement} follow from the definitions of $\gamma_n$ and $\delta_n$. Finally, condition (d) of Lemma \ref{general-statement} follows from the fact that $g$ is a diffeomorphism and $\Xi_n$ is an isometry, so $\norm{A_n} = \norm{\Xi_{n+1} \circ Dg_{x_n} \circ \Xi_n^{-1}}$ is uniformly bounded away from 0. 
			
			We begin by proving summability of $\delta_n$. Assume $\diam P < \rho$, where $\rho$ is the injectivity radius of $M$. Since $y_n \in \gamma^s(x_n)$ and $d(x_n, y_n) < \rho$, the tangent vector $v_n = \left(\exp_{x_n}\right)\big|_{B(\rho,n)}^{-1}(y_n)$ lies in the stable cone $K^-_n \subset T_{x_n}M$, where $B(\rho,n) = \{v \in T_{x_n}M : \norm{v} < \rho\}$. By symmetry of the vector field (\ref{slowdown vector field}), we only need to consider the behavior of the trajectories $\{x_n\}$ and $\{y_n\}$ in the ``upper subsector'' $S^s_{j} \cap S^u_{j}$, corresponding to the first quadrant in coordinates given by $\Phi_{j} \circ \oldphi_0$. (Here we denote $S_j^s$, $S_j^u$, and $\Phi_j$ to be the subsets and functions described earlier as $S_{kj}^s$, $S_{kj}^u$, and $\Phi_{kj}$, where we did not assume we only had one singularity.) Further assume $\tilde s_2 := \im\left(\Phi_{j}(\oldphi_0(y))\right) > s_2 := \im\left(\Phi_{j}(\oldphi_0(x))\right)$, so that $\Delta s_2 := \tilde s_2 - s_2 > 0$. Otherwise, exchange the sequences $\{x_n\}$ and $\{y_n\}$. 
			
			Recall the itinerary $\Is(x) = \left\{0 = n_0 < n_1 < \cdots < n_{2L+1} = \tau(x)\right\} \subset \Z$ of the point $x \in \Lambda$, defined via $x_n \in \Us_0$ if and only if $n_{2j-1} \leq n < n_{2j}$. Consider $n_{2j} \leq n < n_{2j+1}$, so $x_n \not\in \Us_0$. In coordinates, $g(s_1, s_2) = (\lambda s_1, \lambda^{-1}s_2)$, so $A_n = B_n$ are constant matrices, so $\delta_n = 0$. 
			
			Suppose now that $n_{2j+1} \leq n < n_{2j+2}$. Denote by $D(s_1, s_2)$ the coefficient matrix of the variational equations of (\ref{slowdown vector field}), given explicitly by
			\begin{equation}\label{variational}
			D(s_1, s_2) = \log\lambda \left[\begin{array}{cc}
			\Psi_p(u) + 2s_1^2\dot{\Psi}_p(u) & 2s_1 s_2 \dot\Psi_p(u) \\
			-2s_1s_2\dot{\Psi}_p(u) & -\Psi_p(u)-2s_2^2\dot{\Psi}_p(u)
			\end{array}\right].
			\end{equation}
			Let $s(t)$, $\tilde s(t) : \left[n,n+1\right]\to\R^2$ be solutions to (\ref{slowdown vector field}) with initial condition $s(n) = x_n$ and $\tilde s(n) = y_n$, and let $A_n(t)$ and $B_n(t)$ be the $2\times2$ Jacobian matrices 
			\[
			A_n(t) = d(\theta_t)\left(\left(\Phi_{kj}\circ\oldphi_k\right)(x_n)\right) \quad \textrm{and} \quad B_n(t) = d(\theta_t)\left(\left(\Phi_{kj}\circ\oldphi_k\right)(y_n)\right),
			\]
			where $\theta_t : \R^2 \to \R^2$ is the time-$t$ map of the flow of \ref{slowdown vector field} on $\R^2$, for $n \leq t \leq n+1$. Then $A_n(1) = A_n$ and $B_n(1) = B_n$ from before, and $A_n(t)$ and $B_n(t)$ are the unique solutions to the systems of differential equations
			\[
			\frac{dA_n(t)}{dt} = D(s(n+t))A_n(t) \quad \textrm{and} \quad \frac{dB_n(t)}{dt} = D(\tilde s(n+t)) B_n(t)
			\]
			with initial conditions $A_n(0) = B_n(0) = \mathrm{Id}$. It follows that $A_n(t) - B_n(t)$ satisfies the differential equation
			\[
			\frac{dA_n(t)}{dt} - \frac{dB_n(t)}{dt} = \big( D(s(n+t)) - D(\tilde s(n+t))\big)A_n(t) + D(\tilde s(n+t))(A_n(t) - B_n(t)).
			\]
			Using the integrating factor $\exp \int_0^t D(\tilde s(n+\tau)) \, d\tau = B_n(t)$, this implies
			\begin{equation}\label{matrix-integrating-factor}
			A_n(t) - B_n(t) = B_n(t) \int_0^t B_n(t)^{-1}\big( D(s(n+t)) - D(\tilde s(n+t))\big) A_n(t) \, d\tau.
			\end{equation}
			Note $\norm{D(s) - D(\tilde s)} \leq \norm{\del D(\xi)}\norm{\Delta s}$, where $\del D(s)$ denotes the total derivative of the matrix $D(s_1, s_2)$ and $\xi = (\xi_1, \xi_2)$, with $\min\{s_i, \tilde s_i \}\leq \xi_i \leq \max\{s_i, \tilde s_i\}$. This, in conjunction with (\ref{matrix-integrating-factor}) and Lemma \ref{dij-estimate}, gives us: 
			\begin{align*}
			\norm{A_n - B_n} &\leq \norm{B_n(1)} \sup_{0 \leq \tau \leq 1} \norm{B_n(\tau)^{-1}} \norm{A_n(\tau)} \norm{D(s(n+\tau)) - D(\tilde s(n+\tau))} \\
			&\leq \norm{B_n(1)} \sup_{0 \leq \tau \leq 1} \norm{B_n(\tau)^{-1}} \norm{A_n(\tau)} \norm{\del D(\xi(n+\tau))} \norm{\Delta s(n+\tau)} \\
			&\leq C_p \sup_{0 \leq \tau \leq 1} \left(\xi_1^2 + \xi_2^2\right)^{(p-4)/2p}(n+\tau) \norm{\Delta s(n+\tau)}, \numberthis\label{A_n - B_n norm bound}
			\end{align*}
			where $C_p$ is a constant that depends on $p$, but not on $n$ (as the matrices $B_n(t)$ and $A_n(t)$ are uniformly bounded above and below in $n$ and in $t$). %Hence $\norm{A_n - B_n}$ is uniformly bounded in $n$. 
			
			By condition (4) of Lemma \ref{general-statement} and the definition of $\delta_n$, 
			\begin{align*}
			\delta_n &\leq \frac 1{cd(x,y)} \norm{A_n - B_n} = \frac 1 c \frac{d\left(x_{n_{2j+1}}, y_{n_{2j+1}}\right)}{d(x,y)} \frac{\norm{A_n - B_n}}{d\left( x_{n_{2j+1}}, y_{n_{2j+1}}\right)}.
			\end{align*}
			We now claim that
			\begin{equation}\label{Ds_j estimate}
			\Ds_j := \sum_{n=n_{2j+1}}^{n_{2j+2}-1} \frac{\norm{A_n - B_n}}{d\left(x_{n_{2j+1}}, y_{n_{2j+1}}\right)} \leq C, 
			\end{equation}
			where $C$ is a constant independent of $j$. If this is true, then because $\delta_n = 0$ for $n_{2j} \leq n < n_{2j+1}$, by (\ref{inner-slowdown-lipschitz}),
			\begin{align*}
			\sum_{n=0}^{\tau(x)-1} \delta_n &= \sum_{j=1}^L \sum_{n=n_{2j+1}}^{n_{2j+2}-1} \delta_n = \sum_{j=1}^L \frac 1 c \frac{d\left(x_{n_{2j+1}}, y_{n_{2j+1}}\right)}{d(x,y)} \sum_{n=n_{2j+1}}^{n_{2j+2}-1} \frac{\norm{A_n - B_n}}{d\left(x_{n_{2j+1}}, y_{n_{2j+1}}\right)} \\
			&= \frac C c \sum_{j=1}^L \theta_1^j \leq \tilde C,
			\end{align*}
			and because $\theta_1$ is independent of $x, y \in P$, and $c$ and $C$ are both of order $\sup_{n} \norm{A_n}$, $\tilde C$ is also independent of our choice of $x$ and $y$. 
			
			Recall that $\left[m_j^1, m_j^2\right] \subseteq \left[n{2j+1}, n_{2j+2}-1\right]$ is the largest (possibly empty) interval of integers with $x_m \in D_{\tilde r_1}$ for each $n \in \left[m_j^1, m_j^2\right]$, and $\left[m_j^1, T_j\right]$ is the largest time interval for which $s_1(t) \leq s_2(t)$ for all $m_j^1 \leq t \leq T_j$. If $\left[m_j^1, m_j^2\right]$ is empty, then $s(t) \in \left(\Phi_{kj}\circ\oldphi_k\right)\left(D_{\tilde r_0} \setminus D_{\tilde r_1}\right)$ for all $t \in \left[n_{2j+1}, n_{2j+2}-1\right]$. In this instance, by Lemma \ref{annular bound in M}, $n_{2j+2} - n_{2j+1} \leq T$ is uniformly bounded, and hence (\ref{Ds_j estimate}) is a sum of uniformly boundedly many terms that are uniformly bounded, by (\ref{A_n - B_n norm bound}). 
			
			Now suppose $\left[m_j^1, m_j^2\right]$ is nonempty. The sum in (\ref{Ds_j estimate}) splits into four different sums: 
			\begin{equation}\label{we-four-sums}
			\Ds_j = \left( \sum_{n=n_{2j+1}}^{m_j^1 - 1} + \sum_{n = m_j^1}^{T_j-1} + \sum_{n=T_j}^{m_j^2} + \sum_{n=m_j^2+1}^{n_{2j+2}-1}\right)\frac{\norm{A_n - B_n}}{d\left(x_{n_{2j+1}}, y_{n_{2j+1}}\right)} .
			\end{equation}
			We show that each of these sums is themselves uniformly bounded. This is true for the first and fourth sum, because in these instances, $s(t)$ is in the annular region $\left(\Phi_{kj}\circ\oldphi_k\right)\left( D_{\tilde r_0} \setminus D_{\tilde r_1}\right)$, and so the number of summands is uniformly bounded by Lemma (\ref{annular bound in C}). 
			
			To show this for the middle two sums, note that since $\tilde s(t) \in \R^2$ is in the stable cone of $s(t)$ for all $t$ in the domain, we have 
			\begin{equation}\label{Delta s_1 and Delta s_2}
			|\Delta s_1| \leq \alpha \Delta s_2 \leq \Delta s_2. 
			\end{equation}
			First, suppose $m_j^1 \leq n \leq T_j-1$, so that $s_1(t) \leq s_2(t)$. We would like to apply Lemma (\ref{spread-in-slowdown lemma}) in the interval $\left[m_j^1, n\right]$, so we require $\frac{\Delta s_2(m_j^1)}{s_2(m_j^1)} \leq \frac{1-\alpha}{72}$. This is attainable by choosing $r_0$ to be sufficiently small and $Q$ in (\ref{Q-definition}) to be sufficiently large. Applying Lemma (\ref{spread-in-slowdown lemma}) for $n \leq T_j-1$, and $0 \leq \tau \leq 1$, we get:
			\begin{align*}
			|\Delta s(n+\tau)| &\leq 2\Delta s_2(n+\tau) \\
			&\leq 2\frac{\Delta s_2(m_j^1)}{s_2(m_j^1)} s_2(n+\tau) \left(1 + 2^{\frac{p-2}p} C_0 s_2(m_j^1)^{\frac{2p-4}p} (n+\tau - m_j^1)\right)^{-\beta} \\
			&\leq 2\frac{\Delta s_2(m_j^1)}{s_2(m_j^1)} s_2(n+\tau) \left(1 + C_0 s_2(m_j^1)^{\frac{2p-4}p} (n+\tau - m_j^1)\right)^{-\beta}\numberthis \label{spread-lemma-revisited}
			\end{align*}
			since $\beta = 2^{-(3p-2)/p}(1-\alpha) > 0$. Recalling $\xi(t) = (\xi_1(t), \xi_2(t))$ is such that $\min\{s_i, \tilde s_i \}\leq \xi_i \leq \max\{s_i, \tilde s_i\}$ for $i=1,2$, (\ref{xi-est-preT}) gives us
			\[
			s_2^2(t) \leq \left(\xi_1^2 + \xi_2^2\right)(t) \leq 2(1+\kappa)^2s_2^2(t) \leq Cs_2^2(t)
			\]
			as $\kappa = \frac{\Delta s_2}{s_2} \leq \frac{1-\alpha}{72}$. Estimates (\ref{A_n - B_n norm bound}) and (\ref{spread-lemma-revisited}) give us:
			\begin{align*}
			&\norm{A_n - B_n} \\
			&\leq C\frac{\norm{\Delta s(m_j^1)}}{s_2(m_j^1)}\sup_{0 \leq \tau \leq 1} s_2(n+\tau)^{\frac{2p-4}p} \left(1+C_0 s_2(m_j^1)^{\frac{2p-4}p}(n+\tau-m_j^1)\right)^{-\beta},
			\end{align*}
			where we are using the fact that $|\Delta s_2| \leq \norm{\Delta s}$. Applying Lemma \ref{s1 and s2 approx}(b) on the interval $\left[m_j^1, n+1\right]$ gives us 
			\begin{align*}
			&\norm{A_n - B_n} \\
			&\leq C\frac{\norm{\Delta s(m_j^1)}}{s_2(m_j^1)}\sup_{0 \leq \tau \leq 1} s_2(m_j^1)^{\frac{2p-4}p} \left(1+C_0 s_2(m_j^1)^{\frac{2p-4}p}(n+\tau-m_j^1)\right)^{-1-\beta} \\ 
			&=  C\norm{\Delta s(m_j^1)} s_2(m_j^1)^{\frac{p-4}p} \left(1+C_0 s_2(m_j^1)^{\frac{2p-4}p}(n-m_j^1)\right)^{-1-\beta}.
			\end{align*}
			We make three observations. First, recalling that $n=m_j^1$ is the first time that $s(n)$ is within $\tilde r_1$ of the origin, we observe that $s_2(m_j^1)$ is bounded above and below by a constant multiple of $\tilde r_1$, independent of $x \in \Lambda$ or $j = 1, \ldots, L$. Second, $\norm{\Delta s(m_j^1)} = d\left(x_{m_j^1}, y_{m_j^1}\right)$, by definition of our Riemannian metric in $\Us_0$. Third, since Lemma \ref{annular bound in C} implies $m_j^1 - n_{2j+1}$ is bounded by a value independent of $x$ or $j$, the value $\frac{d\big(x_{m_j^1}, y_{m_j^1}\big)}{d\big(x_{2j+1}, y_{2j+1}\big)}$ is uniformly bounded independently of $x, y \in \Lambda$ or $j \geq 1$. These three observations imply: 
			\begin{align*}
			\frac{\norm{A_n - B_n}}{d\left(x_{2n+1}, y_{2n+1}\right)} \leq C \left(1+C_0 s_2(m_j^1)^{\frac{2p-4}p}(n-m_j^1)\right)^{-1-\beta}.
			\end{align*}
			Therefore, 
			\[
			\sum_{n=m_j^1}^{T_j-1} \frac{\norm{A_n - B_n}}{d\left(x_{2n+1}, y_{2n+1}\right)} \leq \sum_{n=m_j^1}^{\infty}C \left(1+C_0 s_2(m_j^1)^{\frac{2p-4}p}(n-m_j^1)\right)^{-1-\beta},
			\]
			which is uniformly bounded in $j$. Therefore the second term in (\ref{we-four-sums}) is uniformly bounded in $j$. 
			
			Finally, we turn our attention to the case where $T_j \leq n \leq m_j^2$, where we have $s_1 \geq s_2$. By symmetry, we have that $T_j \geq \left(m_j^2 + m_j^1 - 2\right)/2$. By (\ref{Delta s_1 and Delta s_2}) and the second inequality in Lemma \ref{spread-in-slowdown lemma}, we have: 
			\begin{align*}
			\lVert \Delta s(n&+\tau)\rVert\leq 2\Delta s_2(n+\tau) \\
			&\leq 2\frac{\Delta s_2(T_j)}{s_1(T_j)} s_1(n+\tau) \left( \frac{1+2^{(p-2)/p}C_0 s_1(m_2^j)^{(2p-4)/p}(m^2_j-n-\tau)}{1+2^{(p-2)/p}C_0 s_1(m_2^j)^{(2p-4)/p}(m^2_j-T_j)}\right)^{\beta}.
			\end{align*}
			%		\newline 
			%		\begin{equation*}
			%		\norm{\Delta s(n+\tau)} \leq 2\Delta s_2(n+\tau) \hspace{4in}
			%		\end{equation*}
			%		\begin{equation*}
			%		\qquad \qquad \leq 2\frac{\Delta s_2(T_j)}{s_1(T_j)} s_1(n+\tau) \left( \frac{1+2^{(p-2)/p}C_0 s_1(m_2^j)^{(2p-4)/p}(m_2^j-n-\tau)}{1+2^{(p-2)/p}C_0 s_1(m_2^j)^{(2p-4)/p}(m_2^j-T_j)}\right)^{\beta}
			%		\end{equation*}
			Since $\min\{s_i, \tilde s_i\} \leq \xi_i \leq \max\{s_i, \tilde s_i\}$ for $i=1,2$, we have $s_i - |\Delta s_i| \leq \xi_i \leq s_i + |\Delta s_i|$. In particular, 
			\begin{align*}
			\xi_1^2 + \xi_2^2 &\geq \xi_1^2 \geq (s_1 - |\Delta s_1|)^2 = s_1^2 \left( 1-\frac{|\Delta s_1|}{s_1}\right)^2 \geq s_1^2\left(1-\frac{\Delta s_2}{s_1}\right)^2 \geq C^{-1}s_1^2,
			\end{align*}
			and
			\[
			\xi_1^2 + \xi_2^2 \leq \left(s_1 + |\Delta s_1|\right)^2 + \left(s_2 + |\Delta s_2|\right)^2 \leq 2\left(s_1 + \Delta s_2\right)^2 = 2s_1\left(1+\frac{\Delta s_2}{s_1}\right)^2 \leq Cs_1^2,
			\]
			which both follow because $\frac{\Delta s_2}{s_1}$ is monotonically decreasing by (\ref{chi-est-T1}). Together, these two estimates imply 
			\[
			\left(\xi_1(n+\tau)^2 + \xi_2(n+\tau)^2\right)^{(p-4)/2p} \leq Cs_1(n+\tau)^{(p-4)/p}.
			\]
			Applying (\ref{A_n - B_n norm bound}) and inequality (a) in Lemma \ref{s1 and s2 approx} to these inequalities gives us:
			\begin{align*}
			&\norm{A_n - B_n} \leq C \sup_{0 \leq \tau \leq 1} \bigg[s_1(n+\tau)^{(p-4)/p} \norm{\Delta s(n+\tau)}\bigg]\\
			&\leq 2C\frac{\Delta s_2(T_j)}{s_1(T_j)} \sup_{0 \leq \tau \leq 1} \left[s_1(n+\tau)^{\frac{2p-4}p}  \left( \frac{1+2^{\frac{p-2}p}C_0 s_1(m^2_j)^{\frac{2p-4}p}(m^2_j-n-\tau)}{1+2^{\frac{p-2}p}C_0 s_1(m^2_j)^{\frac{2p-4}p}(m^2_j-T_j)}\right)^{\beta}\right] \\
			&\leq 2C\frac{\Delta s_2(T_j)}{s_1(T_j)} s_1(m_j^2)^{\frac{2p-4}p}  \sup_{0 \leq \tau \leq 1} \left[\frac{\left( 1+2^{\frac{p-2}p}C_0 s_1(m^2_j)^{\frac{2p-4}p}(m^2_j-n-\tau)\right)^{\beta-1}}{\left(1+2^{\frac{p-2}p}C_0 s_1(m^2_j)^{\frac{2p-4}p}(m^2_j-T_j)\right)^{\beta}}\right].
			\end{align*}
			By (\ref{dkappa-approx-2}), since $s_1(m_j^2)$ and $s_2(m_j^1)$ are uniformly bounded,
			\begin{align*}
			\frac{|\Delta s_2(T_j)|}{s_1(T_j)}s_1(m_j^2)^{(2p-4)/p} &= \frac{|\Delta s_2(T_j)|}{s_2(T_j)}s_1(m_j^2)^{(2p-4)/p} \\ &\leq \frac{|\Delta s_2(m_j^1)|}{s_2(m_j^1)}s_1(m_j^2)^{(2p-4)/p}
			\leq C|\Delta s_2(m_j^1)|.
			\end{align*}
			Furthermore, since $\frac{|\Delta s_2(m_j^1)|}{d\left(x_{n_{2j+1}}, y_{n_{2j+1}}\right)}$ is uniformly bounded, we finally obtain: 
			\[
			\frac{\norm{A_n - B_n}}{d\left(x_{n_{2j+1}}, y_{n_{2j+1}}\right)} \leq C\frac{\left( 1+2^{\frac{p-2}p}C_0 s_1(m^2_j)^{\frac{2p-4}p}(m^2_j-n)\right)^{\beta-1}}{\left(1+2^{\frac{p-2}p}C_0 s_1(m^2_j)^{\frac{2p-4}p}(m^2_j-T_j)\right)^{\beta}}.
			\]
			Therefore, 
			\begin{align*}
			\sum_{n=T_j}^{m_j^2} \frac{\norm{A_n - B_n}}{d\left(x_{n_{2j+1}}, y_{n_{2j+1}}\right)} &\leq C \left(1+2^{\frac{p-2}p}C_0 s_1(m^2_j)^{\frac{2p-4}p}(m^2_j-T_j)\right)^{-\beta} \\
			&\qquad \qquad \times \sum_{n=T_j}^{m_j^2} \left( 1+2^{\frac{p-2}p}C_0 s_1(m^2_j)^{\frac{2p-4}p}(m^2_j-n)\right)^{\beta-1} \\
			&\leq C \left(1+2^{\frac{p-2}p}C_0 s_1(m^2_j)^{\frac{2p-4}p}(m^2_j-T_j)\right)^{-\beta} \\
			&\qquad \qquad \times \left( 1 + \int_{0}^{m_j^2-T_j} \left( 1+2^{\frac{p-2}p}C_0 s_1(m^2_j)^{\frac{2p-4}p}\tau\right)^{\beta-1} \, d\tau \right) \\
			&\leq C \left(1+2^{\frac{p-2}p}C_0 s_1(m^2_j)^{\frac{2p-4}p}(m^2_j-T_j)\right)^{-\beta} \\
			&\qquad \qquad \times \left( 1 + \frac{\left(1+2^{\frac{p-2}p}C_0 s_1(m_j^2)^{\frac{2p-4}p}(m_j^2-T_j)^{\frac{p-2}p}\right)^\beta}{2^{\frac{p-2}p}C_0 s_1(m_j^2)^{\frac{2p-4}p}\beta}\right) \\
			&\leq C\left(1+\left(2^{\frac{p-2}p}\tilde r_1^{\frac{2p-4}p}C_0\beta\right)^{-1}\right),
			\end{align*}
			where the second inequality follows from the fact that the integrand is a decreasing function of $\tau$, and the final inequality follows from the fact that $\tilde r_1 \leq s_1(m_j^2)$ by definition of $m_j^2$. Therefore the third sum of (\ref{we-four-sums}) is uniformly bounded. This completes the proof that $\delta_n$ is a summable sequence. 
			
			We now prove the estimates involving $\gamma_k$. For $n \in \left[n_{2j}, n_{2j+1}-1\right]$, we have $x_n, y_n \not\in \Us_0$, where $Dg_{x_n}$ and $Dg_{y_n}$ are constant hyperbolic linear transformations. For these values for $n$, the maps contract angles uniformly, so there is a $\gamma > 0$ for which $\gamma_n < \gamma < 1$ for all $n$. For $n \in \left[m_j^1, m_j^2\right]$, we have $x_n \in \Us_1$, so applying Lemma \ref{angle-product}, 
			\begin{align*}
			\prod_{n=m_j^1}^{m_j^2-1} \gamma_n &\leq \left(1+C_0 s_2(m_j^1)^{(2p-4)/p}\left(m_j^2 - m_j^1\right)\right)^{-p/(p-2)}\\
			&\leq \left(1+C\left(m_j^2 - m_j^1\right)\right)^{-p/(p-2)},
			\end{align*}
			since $s_2(m_j^1)$ is uniformly bounded. Because the interval of integers $\left[m_j^1, m_j^2\right]$ differs from $\left[n_{2j+1}, n_{2j+2}-1\right]$ by a finite set, and the cardinality of this finite set is uniformly bounded in $j$ by Lemma \ref{annular bound in C}, there is a uniform constant $C' > 0$ for which 
			\[
			\prod_{j=n_{2j+1}}^{n_{2j+2}-1} \gamma_n \leq C'\left( 1+C\left(m_j^2 - m_j^1\right)\right)^{-p/(p-2)} \leq C'.
			\]
			In particular, 
			\begin{equation}\label{theta3}
			\prod_{n=n_{2j}}^{n_{2j+2}-1} \gamma_n \leq C'\gamma^{n_{2j+1}-n_{2j}} < \theta_3,
			\end{equation}
			for some constant $\theta_3 > 0$. The third estimate of the lemma follows.
			
			To prove the second and final estimate of the lemma, we observe that a similar estimate to (\ref{theta3}) may be made with the upper limit replaced with $n_{2j+1}-1$. In particular, for $n_{2j+1} \leq n \leq n_{2j+2}-1$, 
			\[
			\prod_{k={n_{2j+1}}}^n \gamma_j \leq C'\left(1+C(n - n_{2j+1})\right)^{-p/(p-2)}
			\]
			and
			\[
			\prod_{n=n_{2j}}^{n_{2j+1}-1} \gamma_n < \theta_3'
			\]
			for some $\theta_3' > 0$ that is uniformly bounded. Therefore, 
			\begin{align*}
			\sum_{n=0}^{\tau(x)} &\prod_{k=0}^n \gamma_k = \sum_{j=0}^{L(x)} \sum_{n=n_{2j}}^{n_{2j+2}-1} \prod_{k=0}^n \gamma_k = \sum_{j=0}^{L(x)} \left( \prod_{k=0}^{n_{2j}-1} \gamma_k \sum_{n=n_{2j}}^{n_{2j+2}-1} \prod_{k=n_{2j}}^n \gamma_k\right) \\
			&\leq \sum_{j=0}^{L(x)} \left( \theta_3^j\left(\sum_{n=n_{2j}}^{n_{2j+1}-1} \prod_{k=n_{2j}}^{n} \gamma_k + \prod_{k=n_{2j}}^{n_{2j+1}-1} \gamma_k \sum_{n=n_{2j+1}}^{n_{2j+2-1}} \prod_{k=n_{2j+1}}^{n} \gamma_k\right)\right) \\
			&\leq \sum_{j=0}^{L(x)} \left(\theta_3^j\left(\sum_{n=n_{2j}}^{n_{2j+1}-1} \gamma^{n-n_{2j}} + \theta_3' \sum_{n_{2j+1}}^{n_{2j+2}-1} \left(1+C(n-n_{2j+1})\right)^{-p/(p-2)}\right)\right).
			\end{align*}
			Because the two sums in the inner parentheses above are both uniformly bounded, there is a $C'' > 0$ for which
			\[
			\sum_{n=0}^{\tau(x)} \prod_{k=0}^n \gamma_k \leq C''\sum_{j=0}^{L(x)} \theta_3^j,
			\]
			which gives us the second estimate in the lemma. 
		\end{proof}	
		We continue with the proof of the theorem. Observe that $$\left(\Xi_{\tau(x)}^{-1} \circ \prod_{n=0}^{\tau(x)-1} A_n \circ \Xi_0 \right)(v) = D\big(g^{\tau(x)}\big)_x v \quad \forall \: v \in T_x M,$$
		and
		$$\left(P_{\tau(x)}^{-1} \circ \Xi_{\tau(x)}^{-1} \circ \prod_{n=0}^{\tau(x)-1} B_n \circ \Xi_0 \circ P_0\right)(v) = D\big(g^{\tau(x)}\big)_y v \quad \forall\:v \in T_yM.$$
		In particular, since both $\Xi_n$ and $P_n$ are linear isometries for all $n \geq 0$, we have $$\norm{\prod_{n=0}^{\tau(x)-1} A_n \overline v} = \norm{D\left(g^{\tau(x)}\right)_x v} \quad \forall \: v \in T_x M,$$ and $$\norm{\prod_{n=0}^{\tau(x)-1} B_n \overline w} = \norm{D\left(g^{\tau(x)}\right)_y w} \quad \forall\:w \in T_y M,$$ where $\overline v = \Xi_0 v \in \R^2$ and $\overline w = (\Xi_0 \circ P_0) w \in \R^2$. Additionally, for $v \in T_{x_n}M$ and $w \in T_{y_n}M$, 
		\[
		\angle \left(Dg_{x_n}v, \left( P_{n+1} \circ Dg_{y_n}\right)w\right) = \angle \left(A_n \overline v, B_n \overline w\right),
		\]
		where here $\overline v = \Xi_n v$ and $\overline w = (\Xi_n \circ P_n)w$. 
		
		Now, suppose $v \in K^+(x)$ and $w \in K^+(y)$, and once again denote $\overline v = \Xi_0 v$ and $\overline w = (\Xi_0 \circ P_0)w$. Since $P_0 w \in K^+(x)$, Lemmas \ref{general-statement} and \ref{final-step} yield:
		\begin{equation}\label{first-distortion}
		\left| \log\frac{\norm{D\left(g^{\tau(x)}\right)_xv}}{\norm{D\left(g^{\tau(x)}\right)_yw}}\right| = \left| \log\frac{\norm{\prod_{n=0}^{\tau(x)-1}A_n \overline v}}{\norm{\prod_{n=0}^{\tau(x)-1}B_n \overline w}}\right| \leq C \tilde C\big(d(x,y) + \angle\left( v, P_0 w\right)\big) 
		\end{equation}
		where we are using the fact that $ \angle\left( v, P_0 w\right) = \angle\left(\overline v, \overline w\right)$. Furthermore, for $v \in T_xM$ and $w \in T_yM$, the definition of $\gamma_n$ and Lemma \ref{final-step} give us:
		\begin{align*}
		&\frac{\angle\left(D\left(g^{\tau(x)}\right)_{x}v, \left(P_{\tau(x)} \circ D\left(g^{\tau(x)}\right)_{y}\right)w\right)}{\angle(v, P_0 w)} \\
		&\qquad \qquad = \prod_{n=0}^{\tau(x)-1} \frac{\angle\left(Dg_{x_n}\left(Dg_x^nv\right),\left(P_{n+1} \circ Dg_{y_n} \right)\left(Dg_y^n w\right)\right)}{\angle\left(Dg_x^n v, P_n\left(Dg_n^n w\right) \right) } \\
		&\qquad \qquad = \prod_{n=0}^{\tau(x)-1} \frac{\angle\left(A_n \left(\Xi_n\left(Dg_x^n v\right)\right), B_n\left(\left(\Xi_n \circ P_n\right)\left(Dg_y^n w\right)\right)\right)}{\angle\left(\Xi_n\left(Dg_x^nv\right), \left(\Xi_n \circ P_n\right)\left(Dg_y^n w\right)\right)} \\ \\
		&\qquad \qquad \leq \prod_{n=0}^{\tau(x)-1} \gamma_n \leq \theta_2. \numberthis\label{theta-2}
		\end{align*}
		Denote $\hat G : \Lambda \to \Lambda$ by $\hat G(x) = g^{\tau(x)}(x)$. If $v^n \in E^u\left(\hat G^n(x)\right)$ and $w^n \in E^u\left(\hat G^n(y)\right)$, then there are $v \in E^u(x)$ and $w \in E^u(y)$ such that $v^n = D\hat G^n_x v$ and $w^n = D\hat G^n_y w$. By (\ref{first-distortion}), (\ref{theta-2}), and condition (Y3), 
		\begin{align*}
		\left| \log\frac{\norm{D\hat G_{\hat G^n(x)}v^n}}{\norm{D\hat G_{\hat G^n(y)}w^n}}\right| &\leq C\tilde C\Bigg(d\bigg(\left(g^{\tau(x)}\right)^n(x), \left(g^{\tau(x)}\right)^n(y)\bigg) \\
		&\qquad\qquad+ \angle\left(D\left(g^{\tau(x)}\right)^n_x v, P_{\tau(x)}D\left(g^{\tau(x)}\right)^n_y w\right)\Bigg) 
		\\
		&\leq C\tilde C\big(a^n d(x,y) + \theta_2^n \angle\left(v, P_0 w\right)\big).
		\end{align*}
		Since $0 < a, \theta_2 < 1$, this proves (Y4)(a). 
	\end{proof}
	
	\section{Proof of Theorem \ref{main-theorem}}
	
	We now drop our assumption that the pseudo-Anosov diffeomorphism $g$ admits only one singularity. By Proposition \ref{Young tower nuke} and Theorem \ref{PAD tower}, since $g : M \to M$ is a Young's diffeomorphism, the geometric potential $\phi_1(x) = -\log\left|Dg|_{E^u(x)}\right|$ admits an equilibrium measure, which is the unique $g$-invariant SRB measure. This is the same measure as $\mu_1$ introduced in Proposition \ref{PAD measure}, as $\mu_1$ is absolutely continuous along the unstable foliations and thus an SRB measure. (This justifies our use of the notation $\mu_1$ to describe this measure). 
	
	By Proposition \ref{f-g conjugacy}, the pseudo-Anosov homeomorphism $f$ and the pseudo-Anosov diffeomorphism $g$ possess the same topological and combinatorial data, including topological entropy. Thus the number $S_n$ of $s$-sets $\Lambda_i^s \subset \Lambda$ with inducing time $\tau_i = n$ for $g$ is the same for both $f$ and $g$. Therefore by Lemma \ref{number-of-inducing-sets}, there is an $h < \htop(g) = \htop(f)$ such that $S_n \leq e^{hn}$. 
	
	Recall that $\nu$ is the measure on $M$ given locally by the product of lengths of local stable and unstable leaves described in Theorem \ref{PAH measure}, and $\mu_1$ is the measure given by the Riemannian metric $\zeta$ described in Proposition \ref{PAD measure}. By Theorem \ref{PAH measure}, $\nu$ has a density with respect to $\mu_1$, which vanishes at the singularities. By Proposition 10.13 and Lemma 10.22 of \cite{FLP79}, $h_\nu(f) = \htop(f) = \log\lambda$, so in fact $h < h_\nu(f)$. Since $\nu = \mu_1$ on $M \setminus \Us_0$, and $\mu_1(\Us_0)$ may be made arbitrarily small by shrinking $r_0$ if necessary, the Pesin entropy formula implies 
	\begin{align*}
	h_{\mu_1}(g) &= \int_M \log \left| Dg|_{E^u(x)}\right| \, d\mu_1(x) \\
	&= \int_{M \setminus \Us_0} \log\lambda \: d\nu + \int_{\Us_0}  \log \left| Dg|_{E^u(x)}\right| \, d\mu_1(x) < h_\nu(f) + \epsilon, \numberthis \label{entropy-difference}
	\end{align*}
	where $\epsilon>0$ is as small as we need. From this we conclude that $h < h_{\mu_1}(g)$. Hence by Proposition \ref{Young tower nuke}, there is a $t_0 < 0$ for which for all $t \in (t_0, 1)$, there is a measure $\mu_t$ on $P$ that is an equilibrium state for the geometric $t$-potential $\phi_t$. 
	
	Since $f$ is Bernoulli, every power of $f$ is ergodic, so $f$ satisfies the arithmetic condition. Since $f$ and $g$ are topologically conjugate, this is also true for $g$. 
	
	We now prove (\ref{4.2 bound}). If $x, y \in \Lambda_i^s$ and $y \in \gamma^s(x)$, the distance $d\left(f^j(x) f^j(y)\right)$ decreases with $j$. On the other hand, if $y \in \gamma^u(x)$, then $d\left(f^j(x), f^j(y)\right)$ increases with $j$, but is bounded by $\diam\,P$ when $j = \tau(x)$. An application of the triangle inequality and hyperbolic product structure of $\Lambda$ now yields (\ref{4.2 bound}). It now follows that $\mu_t$ has exponential decay of correlations and satisfies the Central Limit Theorem, by Proposition \ref{Young tower nuke}. Since $(M, g, \mu_t)$ has exponential decay of correlations, this dynamical system is mixing. By Theorem 2.3 in \cite{FZTowers}, $(M, g, \mu_t)$ is Bernoulli. 
	
	To show $r_0$ may be chosen to accommodate any $t_0$, we show that as $r_0 \to 0$, we may take $t_0 \to -\infty$. Fix $\epsilon > 0$, and choose $x \in \Lambda_i^s$. Recall $g = f$ outside of $\tilde\Us_0$; in particular, the local stable and unstable leaves are unchanged outside of $\tilde \Us_0$. Assume $x$ is a generic point for the SRB measure $\mu_1$. Let $\tilde \Us_2 = \union_{k=1}^m \oldphi_k^{-1}\left(D_{\tilde r_1/4}\right)$, and write $\tau_i$ as
	\[
	\tau_i = \sum_{j=1}^s n_j,
	\]
	where the integers $n_j$ are chosen like so: 
	\begin{itemize}
		\item The integer $n_1$ is the first time when $g^{n_1}(x) \in \tilde \Us_0 \setminus \tilde\Us_2$; 
		\item The integer $n_2$ is the first time after $n_1$ when $g^{n_1 + n_2}(x) \in \tilde\Us_2$; 
		\item the number $n_3$ is the first time after $n_1 + n_2$ when $g^{n_1 + n_2 + n_3}(x) \in \tilde\Us_0 \setminus \tilde\Us_2$;
		\item the number $n_4$ is the first time after $n_1 + n_2 + n_3$ when $g^{n_1 + n_2 + n_3 + n_4}(x) \not\in \tilde\Us_0$;
	\end{itemize}
	and so on. It is possible that some $n_j$ may be equal to 0, but this does not change our calculations. Observe $Q \leq n_1$, where $Q$ is the number from (\ref{Q-definition}). If $r_0$ is sufficiently small, $Q$ is large enough to ensure that 
	\begin{equation}\label{n_1 est}
	\log\left|J^u g^{n_1}(x)\right| \leq n_1(\log\lambda + \epsilon).
	\end{equation}
	By (\ref{variational}), for $x \in \tilde\Us_0 \setminus \tilde\Us_2$, we have $\log \left|J^u g(x)\right| \leq \log N$ for some constant $N$ independent of $r_0$ or of the number of prongs $p$. Therefore, 
	\begin{equation}\label{n_2 est}
	\log\left|J^u g^{n_2}(x)\right| \leq n_2 \log N \quad \textrm{and} \quad \log\left|J^u g^{n_4}(x)\right| \leq n_4 \log N.
	\end{equation}
	For $x \in \tilde\Us_2$, if $x$ is in a neighborhood of a singularity with $p$ prongs, $\Psi_p(u) = \left( \frac p 2\right)^{(2p-4)/p} u^{(p-2)/p}$ and $\dot{\Psi}_p(u) = \frac{p-2}p \left(\frac p 2\right)^{(2p-4)/p} u^{-2/p}$. By (\ref{variational}), for such points $x$, $\log \left| J^u g(x)\right|\leq \log\lambda$. Therefore, 
	\begin{equation}\label{n_3 est}
	\log\left| J^u g^{n_3}(x)\right| \leq n_3 \log\lambda.
	\end{equation}
	Similar estimates hold for the other $n_j$. Observe that 
	\begin{equation}\label{J^uF sum}
	\log\left|J^u \hat G(x)\right| \leq \sum_{j=1}^s \log\left|J^u g^{n_1 + \cdots + n_j}\left(g^{n_1 + \cdots + n_{j-1}}(x)\right)\right|.
	\end{equation}
	Similarly to Lemma \ref{annular bound in M}, the number of iterates the orbit of $x$ spends in $\hat\Us_0 \setminus \hat\Us_2$ is bounded above by a constant $T_0'$ independent of both $r_0$ and $p$. It follows from (\ref{n_1 est})-(\ref{J^uF sum}) and the definition of $\lambda_1$ in (\ref{lambda_1 def}) that
	\[
	\log\lambda_1 \leq \log\lambda + \epsilon + \frac{2T_0' \log N}{Q} \leq \log\lambda + 2\epsilon.
	\]
	Meanwhile, (\ref{entropy-difference}) implies that for sufficiently small $r_0$, 
	\begin{equation}\label{2.1 statement 4}
	\left| \int_M \log\left|Dg|_{E^u(x)}\right|\,d\mu_1(x) - \log\lambda\right| < \epsilon,
	\end{equation}
	or equivalently, 
	\[
	\log\lambda - \epsilon \leq h_{\mu_1}(g) \leq \log\lambda+\epsilon.
	\]
	Furthermore, one can show $\log\lambda_1 \geq h_{\mu_1}(g)$ (see Remark 3 in \cite{PSZ17}, which is a general statement about Young diffeomorphisms). Therefore, 
	\[
	\log\lambda - \epsilon \leq h_{\mu_1}(g) \leq \log\lambda_1 \leq \log\lambda+2\epsilon.
	\]
	It follows that the difference $\log\lambda_1 - h_{\mu_1}(g)$ can be made arbitrarily small if $r_0$ is chosen to be sufficiently small. By (\ref{t_0 def}), this shows that $t_0 \to -\infty$ as $r_0 \to 0$. 
	
	We now show how $\mu_t$ may be extended to a measure on $M$, as opposed to a measure only on images of the base of the tower. Suppose we have another element $\tilde P$ of the Markov partition satisfying (\ref{Q-definition}). As above, there is a $\tilde t_0 = t_0(\tilde P) < 0$ for which for every $t \in (\tilde t_0, 1)$, there is a unique equilibrium state $\tilde\mu_t$ for the geometric $t$-potential among all measures $\mu$ for which $\mu(\tilde P) > 0$, and $\tilde\mu_t(U) > 0$ for all open sets $\tilde U \subset P$. Since $g$ is topologically conjugate to a Bernoulli shift, $g$ is topologically transitive. Therefore for any open sets $\tilde U \subset \tilde P$ and $U \subset P$, there is an integer $k \geq 0$ for which $g^k(\tilde U) \cap U \neq \emptyset$. By invariance of $\tilde\mu_t$ and $\mu_t$ under $g$, it follows that $\mu_t = \tilde \mu_t$. 
	
	Consider now an element of the Markov partition that does not satisfy (\ref{Q-definition}). If $r_0$ is sufficiently small, the union of all partition elements satisfying (\ref{Q-definition}) form a closed set $Z \subset M$, whose complement is a neighborhood of the singular set $S$ with each component containing a single singularity. If $\omega$ is a $g$-invariant probability measure that does not give weight to partition elements in $Z$, then $\omega$ is a convex combination of the $\delta$-measures concentrated at the singularities. If $P$ is our partition element in the proof of Theorem \ref{PAD tower}, we observe $\omega(P) = 0$, so $\omega$ is clearly out of consideration as an equilibrium measure for $\phi_t$. So any equilibrium measure for $(M,g)$ must charge partition elements in $Z$. Therefore, set 
	\[
	t_0 = \max_{P \in \Ps, \, P \cap Z \neq \emptyset} t_0(P).
	\]
	Since $t_0 \to -\infty$ as $r_0 \to 0$ and $\mu_t(P) > 0$ for $t_0 < t < 1$, this $t_0$ suffices for the first statement of Theorem \ref{main-theorem}. 
	
	To prove Statement 2 of Theorem \ref{main-theorem}, suppose $\omega$ is an invariant ergodic Borel probability measure. By the Margulis-Ruelle inequality, 
	\[
	h_\omega(g) \leq \int_{M} \log\left| Dg|_{E^u(x)}\right|\,d\omega(x) = -\int_M \phi_1\,d\omega.
	\]
	Hence $h_\omega(f) + \int \phi_1\,d\omega \leq 0$. If $\omega$ has only 0 as a nonnegative Lyapunov exponent almost everywhere, then $\log\left|Dg|_{E^u(x)}\right| = 0$ $\omega$-a.e. The only point at which $\log\left|Dg|_{E^u(x)}\right|=0$ is at the singularities of $g$, so $\omega$ is a convex combination of the $\delta$-measures at the singularities. In this instance, we have $h_{\omega}(g) + \int \phi_1\,d\omega = 0$, so $P(\phi_1) = 0$, and $\omega$ is an equilibrium state for $\phi_1$. 
	
	%EDIT: SRB, not $SRB$. 
	
	On the other hand, part 1 of Proposition \ref{Young tower nuke} guarantees the existence of an SRB measure $\mu_1$ for $g$. In particular, $\mu_1$ is a smooth measure, so by the Pesin entropy formula, $h_\mu(f) + \int \phi_1 \,d\mu = 0$, so $\mu$ is also an equilibrium measure. Any other equilibrium measure with positive Lyapunov exponents also satisfies the entropy formula. By \cite{LY84}, such a measure is also an SRB measure, and by \cite{RH2TU}, this SRB measure is unique. This proves Statement 2. 
	
	%EDTIT: $g$, not $f$. 
	
	Finally, to prove Statement 3 of Theorem \ref{main-theorem}, fix $t > 1$, and let $\omega$ be an ergodic measure for $g$. Again, by the Margulis-Ruelle inequality, $$h_\omega(g) \leq t\int \log\left|Dg|_{E^u(x)} \right|\,d\omega,$$ with equality if and only if $\int \log\left|Dg|_{E^u(x)} \right|\,d\omega = 0$. In particular, we have equality if and only if $\omega$ has zero Lyapunov exponents $\omega$-a.e. As we saw, the only measures satisfying this are convex combinations of $\delta$-measures at singularities, so $h_\omega(g) + \int \phi_t \,d\omega \leq 0$, with equality only for $\omega=\sum \lambda_i \delta_{x_i}$, with $\sum \lambda_i = 1$. Hence the only equilibrium states for $\phi_t$ with $t > 1$ are convex combinations of $\delta$-measures at singularities.

		\begin{Backmatter}
			
			\begin{ack}
				 I would like to thank Penn State University and the Anatole Katok Center for Dynamical Systems and Geometry where this work was done. I also thank my advisor, Y. Pesin, for introducing me to this problem and for valuable input over the course of my investigation into pseudo-Anosov systems. 
			\end{ack}

		\end{Backmatter}
		
	\end{document}